\documentclass[12pt]{amsart}

\usepackage{color}
\usepackage{amsfonts}
\usepackage{amscd}
\usepackage{amssymb}
\usepackage{amsthm}
\usepackage{amsmath}
\usepackage{mathrsfs}
\usepackage{comment}
\usepackage{epsfig}
\usepackage[all]{xy}

\usepackage[colorlinks=true, pdfstartview=FitV, linkcolor=blue, citecolor=blue, urlcolor=blue]{hyperref}
\usepackage{tikz,etex,bbm,xspace,hyperref}
\newcommand{\arxiv}[1]{\href{http://arxiv.org/abs/#1}{\tt arXiv:\nolinkurl{#1}}}

\newtheorem{Theorem}{Theorem}[section]
\newtheorem{Proposition}[Theorem]{Proposition}
\newtheorem{Lemma}[Theorem]{Lemma}
\newtheorem{Nested Lemma}[Theorem]{Nested Lemma}

\newtheorem{Corollary}[Theorem]{Corollary}

\theoremstyle{definition}

\newtheorem{Definition}[Theorem]{Definition}

\newtheorem{Remark}[Theorem]{Remark}

\renewcommand{\theenumi}{\roman{enumi}}
\renewcommand{\labelenumi}{(\theenumi)}

\makeatletter
\renewcommand{\@makefnmark}{\mbox{\textsuperscript{}}}
\makeatother

\newcommand{\cL}{\mathcal{L}}
\newcommand{\HN}{\text{HN}}

\newcommand{\wt}{\text{wt}}

\newcommand{\lbc}{\mathbf{c}^\ell}
\newcommand{\rbc}{\mathbf{c}^r}
\newcommand{\ba}{\mathbf{c}}

\newcommand{\cc}{\mathbf{c}}
\newcommand{\dd}{\mathbf{d}}

\newcommand{\cB}{\mathcal{B}}
\renewcommand{\cL}{\mathcal{L}}

\newcommand{\Z}{\mathbb{Z}}

\newcommand{\U}{\mathbf{U}}
\newcommand{\e}{\tilde{e}}
\newcommand{\f}{\tilde{f}}
\newcommand{\tpsi}{\tilde{\psi}}
\newcommand{\cA}{\mathcal{A}}
\newcommand{\ATT}{A_2^{(2)}}

\newcommand{\nc}{\newcommand}
\nc{\cosoc}{\operatorname{cosoc}}
\nc{\soc}{\operatorname{soc}}
\nc{\asl}{\widehat{\mathfrak{sl}}}
\nc{\g}{\mathfrak{g}}
\nc{\br}{\Bbb R}
\nc{\bz}{\Bbb Z}
\nc{\bn}{\Bbb N}
\nc{\omu}{\overline{\mu}}
\nc{\olambda}{\overline{\lambda}}
\nc{\oa}{\overline{a}}
\nc{\Irr}{\text{Irr }}
\nc{\oT}{\overline{T}}
\nc{\oR}{\overline{R}}
\nc{\oI}{\overline{I}}
\nc{\bfv}{{\bf v}}
\nc{\vareps}{\varepsilon}
\nc{\MV}{\mathcal{MV}}
\nc{\be}{\beta}
\nc{\End}{\text{End}}

\begin{document}

\title{Affine PBW Bases and MV Polytopes in rank 2}

\author{Dinakar Muthiah, Peter Tingley}

\address{Dinakar Muthiah, Dept. of Mathematics, Brown University, Box 1917 Providence, RI 02912}
\email{dmuthiah@math.brown.edu}

\address{Peter Tingley, Dept. of Maths and Stats, Loyola University Chicago}
\email{ptingley@luc.edu}

\thanks{The second author was supported by NSF grants DMS-0902649, DMS-1162385 and DMS-1265555.}

\begin{abstract}
Mirkovi\'c-Vilonen (MV) polytopes have proven to be a useful tool in understanding and unifying many constructions of crystals for finite-type Kac-Moody algebras. These polytopes arise naturally in many places, including the affine Grassmannian, pre-projective algebras, PBW bases, and KLR algebras. There has recently been progress in extending this theory to the affine Kac-Moody algebras. A definition of MV polytopes in symmetric affine cases has been proposed using pre-projective algebras. In the rank-2 affine cases, a combinatorial definition has also been proposed. Additionally, the theory of PBW bases has been extended to affine cases, and, at least in rank-2, we show that this can also be used to define MV polytopes. The main result of this paper is that these three notions of MV polytope all agree in the relevant rank-2 cases. Our main tool is a new characterization of rank-2 affine MV polytopes.
\end{abstract}

\maketitle


\section{Introduction}

Let $\g$ be a symmetrizable Kac-Moody algebra. To each irreducible lowest weight representation $V(-\lambda)$ Kashiwara associates a ``crystal" $B(-\lambda)$, which is a combinatorial object that records leading order behavior of the representation. The crystals $B(-\lambda)$ form a directed system whose limit $B(-\infty)$ can be thought of as the crystal for $U^+(\g)$. The foundations of this theory makes heavy use of the quantized universal enveloping algebra associated with $\g$, but the resulting crystals can often be realized in other ways. 

When $\g$ is of finite type, one interesting realization of $B(-\infty)$ is via the theory of MV (Mirkovi\'c-Vilonen) polytopes. The MV polytope $MV_b$ associated to an element $b \in B(-\infty)$ arises in (at least) four natural ways:

\begin{enumerate}

\item \label{int:mv} The original construction as the moment map image of a certain MV cycle in the affine Grassmannian, due to Anderson \cite{And} and studied by Kamnitzer \cite{Kam, KamCrystal}.

\item \label{int:preproj} As the convex hull of the dimension vectors of all subrepresentations of a certain preprojective algebra representation. This comes from work of Baumann-Kamnitzer \cite{BK??} on quiver varieties.

\item \label{int:pbw} As the convex hulls of the paths defined by the PBW monomials corresponding to $b$ with respect to all reduced expressions for the longest word $w_0$ in the Weyl group. This is a rewording of results in \cite{Kam}, which in turn makes use of Lusztig's explicit calculations in \cite{LusPiecewiseLinear}. 

\item \label{int:klr} As the character-polytope of an irreducible representation of a Khovanov-Lauda-Rouquier algebra. This is done by the second author and Webster in \cite{TW}, building on ideas of Kleshchev and Ram \cite{KlRa}.

\end{enumerate}
\noindent In addition, MV polytopes can be approached from a purely combinatorial point of view, using the following result from \cite{Kam}:
\begin{enumerate}
  \setcounter{enumi}{4}
  \item \label{int:tcp} The MV polytopes are exactly those convex polytopes all of whose edges are parallel to roots, and all of whose 2-faces satisfy a combinatorial condition called a ``tropical Pl\"ucker relation."
\end{enumerate}

The current paper is part of an effort to develop a version of this story in affine type.
There is not yet a true affine analogue of \eqref{int:mv}, although see \cite{BFG} for a definition of (open) MV cycles
in untwisted affine types, and \cite{Muth, NSS} for some related combinatorial results. 
The other constructions are to various extents understood in the affine situation:
\begin{enumerate}

\item[\eqref{int:preproj}] The first definition of MV polytopes in symmetric affine types was given in \cite{BKT} using quiver varieties. In that work, an MV ``polytope" is actually a decorated polytope; in addition to the underlying polytope, one must include the data of a partition (of some $N$) associated to each co-dimension 1 face parallel to the imaginary root $\delta$ (this is a rewording of \cite{BKT}, see \cite{TW} for this statement). As in finite type the underlying polytope is the convex hull of the dimension vectors of all subrepresentations of a certain preprojective algebra representation. 

\item[\eqref{int:pbw}] Affine PBW type crystal bases were defined in \cite{BCP,Aka,BN}. Prior to the present work, no connection has been made with affine MV polytopes, although the format of the combinatorics is similar (such as partitions associated to imaginary root directions). 

\item[\eqref{int:klr}] In \cite{TW}, KLR algebras are used to give a construction and combinatorial characterization of MV polytopes in all affine types. That work makes use of the current paper to describe how the affine MV polytopes coming from KLR algebras are related to those coming from quiver varieties in \cite{BKT}.

\item[\eqref{int:tcp}] It is also shown in \cite{BKT, TW} that affine MV polytopes can be characterized by conditions on 2-faces, which reduces the problem of providing a combinatorial description to understanding the rank 2 affine cases. See \cite{BKT} for the symmetric case, and \cite{TW} for the general type affine case. 

\end{enumerate}
Finally, in \cite{BDKT:??}  a combinatorial definition of MV polytopes for the two rank-2 affine cases was given, but 
it was not determined whether this agreed with \cite{BKT} in the relevant case of $\asl_2$. 

This leaves several important questions unanswered. First, one would like to show that the combinatorics in \cite{BDKT:??} matches the rank 2 case of the construction in \cite{BKT}. Second, one would like to relate affine MV polytopes with affine PBW bases. Third, one would like to relate affine MV polytopes to geometry along the geometry in \cite{BFG}. 
The current paper addresses the first question and the rank two case of the second. 

Our result showing that  the combinatorics in \cite{BDKT:??} matches the rank 2 case of the construction in \cite{BKT} is the last step in providing a combinatorial characterization of symmetric affine MV polytopes.

The relationship we describe between rank 2 affine MV polytopes and affine PBW bases essentially answers the rank 2 case of the problem posed by Beck-Nakajima in \cite[Remark 3.29]{BN}. In rank 2 there are exactly two PBW bases, and each basis is in bijection with the crystal $B(-\infty)$. Hence there is a natural bijection between these two bases. 
Answering Beck-Nakajima's question amounts to giving an explicit description of this bijection. Our results show that this bijection is given exactly by the combinatorics in \cite{BDKT:??}.
This should be thought of as a rank-2 affine analogue of Lusztig's piecewise-linear bijections from \cite{LusPiecewiseLinear}.
We believe our methods can be used to relate affine MV polytopes and PBW bases more generally, and thus to answer Beck-Nakajima's question in all affine types. We plan to address this in a future work.

Our main tool 
is a new characterization of rank-2 affine MV polytopes, or more precisely of the map that takes an element of $B(-\infty)$ to its MV polytope. We show that this is the unique map to decorated polytopes satisfying 
some  conditions related to the 
crystal operators and 
Saito's crystal reflections, and one other condition to handle the imaginary roots. For the $\asl_2$ cases, we consider three a priori different maps from $B(-\infty)$ to decorated polytopes: the combinatorial construction from \cite{BDKT:??}, the geometric construction using quiver varieties from \cite{BKT}, and the algebraic construction from PBW bases. We show that all three maps satisfy the conditions of our uniqueness theorem. In particular, they all agree. In the $\ATT$ case, the first two of these constructions make sense, and we show that these also agree.

This paper is organized as follows: In \S\ref{sec:background} we briefly give some general background on quantum affine algebras and crystals. In \S\ref{sec:rank2} we recall the combinatorial definition of rank-2 affine MV polytopes from \cite{BDKT:??} and prove our characterization theorem (Theorem \ref{th:unique-aff}). In \S\ref{sec:pbw} we present some background on affine PBW bases, define a map from these to decorated polytopes, and prove that the resulting PBW polytopes are in fact identical to the combinatorially defined MV polytopes (Theorem \ref{PBWPolytopeTheorem}). In \S\ref{sec:quiver} we recall the construction of $\asl_2$ MV polytopes from \cite{BKT}, show that the result (after a minor change of conventions) also agrees with the combinatorial construction (Theorem \ref{th:sl2=sl2}), and state the final combinatorial characterization of MV polytopes in all symmetric affine types (Theorem \ref{th:comb-poly}). Note that \S\ref{sec:pbw} and \S\ref{sec:quiver} are completely independent of one another.  

\subsection{Acknowledgements} 
We thank Pierre Baumann, Jonathan Beck, Alexander Braverman, Joel Kamnitzer and Hiraku Nakajima for many useful discussions.

\section{Background} \label{sec:background}

\subsection{Quantum Affine Algebras and Canonical Basis} \label{ss:qg}
Our notation for quantum affine algebras will mostly follow \cite[\S2.2]{BN}, and we refer the reader there for more details. Here we only recall the key points, and we will only consider the quantum affine algebras for $\asl_2$ and $\ATT$, which have rank-2 Dynkin diagrams. We label the nodes of those Dynkin diagram by $0,1$ (thought of as integers mod 2), where for $A_2^{(2)}$ the node $0$ corresponds to the long root. Throughout:

\begin{itemize}

 \item $\g$ is the Kac-Moody algebra $\asl_2$ or $\ATT$.

\item $\U$ is the quantized universal enveloping algebra for $\asl_2$ or $\ATT$.

\item $E_0, E_1, F_0, F_1$ are the standard Chevalley generators. 

\item $\U^+$ is the subalgebra of $\U$ generated by $E_0, E_1$. 

\item $\alpha_0$ and $\alpha_1$ are the simple roots for $\g$. 

\item $P$ is the weight lattice of $\g$. 

\item Kashiwara's involution $*$ is the algebra anti-involution of $\U$ which fixes all the Chevalley generators. Notice that $*$ preserves $\U^+$. 

\item For $\asl_2$, $\tau$ is the algebra involution of $\U$ induced by the Dynkin diagram automorphism. $\tau$ also preserves $\U^+$. 

\item $T_0, T_1$ are the generators of the braid group acting on $\U$ (which in the rank two affine cases is just the free group on two generators). By e.g. \cite[Corollary 1.3.3]{Saito:1994}, $T_i \circ * = * \circ T^{-1}_i$.

\item $\cB$ is Lusztig's canonical basis (equivalently Kashiwara's global crystal basis) of $\U^+$. 

\item $\cA$ is the ring consisting of all rational functions in ${\Bbb C}(q)$ which are regular at $q=\infty$, and $\mathcal{L}= \text{span}_\cA \cB$. Recall that $\cB$ descends to a crystal basis of $\cL/q^{-1} \cL$.

\item $\e_i,\f_i$ are Kashiwara's crystal operators on $\cB$ (as a crystal basis for $\cL/q^{-1} \cL$). 

\end{itemize}

\subsection{Crystals} \label{ss:crystals}
We are interested in the crystal $B(-\infty)$ associated with $\U^+(\g)$. This section contains a brief review of some properties of this object, and we refer the reader to e.g. \cite{Kashiwara:1995} or \cite{HongKang} for more details. Recall that $B(-\infty)$ is a set along with operators $\e_i, \f_i: B(-\infty) \rightarrow B(-\infty) \sqcup \{ \emptyset \}$ for each $i \in I$. These satisfy 

\begin{itemize}
\item $B(-\infty)$ has a unique element $b_0$ such that $\f_i(b_0)=0$ for all $i$. 

\item Every element of $B(-\infty)$ can be obtained from $b_0$ by applying some sequence of operators $e_i$ for various $i$. 

\item For $b, b' \in B(-\infty)$, $\e_i(b)=b'$ if and only if $\f_i(b')=b$. 

\item There is a unique map $\wt: B(-\infty) \rightarrow P$ such that $\wt(b_0)=0$ and each operator $\e_i$ has weight $\alpha_i$, the corresponding simple root. 
\end{itemize}
There are also important functions $\varphi_i, \vareps_i : B(-\infty) \rightarrow \bz$ defined by
\begin{equation}
\varphi_i(b) = \max \{ n \in \bz_{\geq 0} : f_i^n b \neq \emptyset \}, \quad \text{and}  \quad \vareps_i(b) = \varphi_i(b)- \langle \alpha_i^\vee, \wt(b) \rangle.
\end{equation}
Finally, there is a weight preserving involution $*: B(-\infty) \rightarrow B(-\infty)$ induced by Kashiwara's involution on $\U^+$.  
Define $\f_i^* = * \f_i *,\e_i^*=* \e_i *, \varphi_i^*= \varphi \circ *$ and $\vareps_i^*= \vareps \circ *.$

By \cite[Proposition 3.2.3]{KS:1997}, for every $i$ and every $b \in B(-\infty)$, the subset of $B(-\infty)$ generated by $b$ by $\tilde e_i, \tilde e^*_i, \tilde f_i$ and $\tilde f^*_i$ looks like

\vspace{0.15cm}

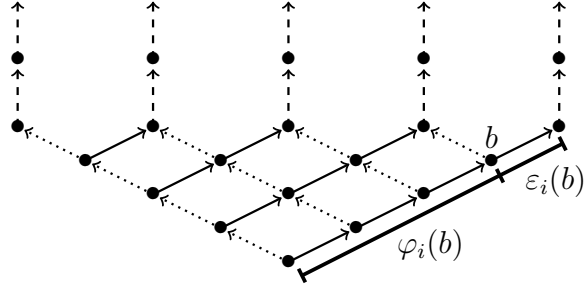
\begin{equation} \label{ii*-pic}
\setlength{\unitlength}{0.15cm}
\begin{tikzpicture}[xscale=0.45,yscale=-0.45, line width = 0.03cm]

\draw[line width = 0.05cm] (10.3,5.2)--(10.5, 5.6);

\draw [line width = 0.05cm] (16.15,2.25)--(16.35, 2.65);

\draw [line width = 0.05cm] (18,1.3)--(18.2, 1.7);

\draw[line width = 0.05cm] (10.4,5.4)--(18.1, 1.5);

\draw node at (14.2, 4.5) {$\varphi_i(b)$};
\draw node at (17.9, 2.6) {$\varepsilon_i(b)$};
\draw node at (16, 1.3) {$b$};

\draw node at (10,5) {$\bullet$};

\draw node at (8,4) {$\bullet$};
\draw node at (12,4) {$\bullet$};

\draw node at (6,3) {$\bullet$};
\draw node at (10,3) {$\bullet$};
\draw node at (14,3) {$\bullet$};

\draw node at (4,2) {$\bullet$};
\draw node at (8,2) {$\bullet$};
\draw node at (12,2) {$\bullet$};
\draw node at (16,2) {$\bullet$};
\draw node at (2,1) {$\bullet$};
\draw node at (6,1) {$\bullet$};
\draw node at (10,1) {$\bullet$};
\draw node at (14,1) {$\bullet$};
\draw node at (18,1) {$\bullet$};

\draw node at (2,-1) {$\bullet$};
\draw node at (6,-1) {$\bullet$};
\draw node at (10,-1) {$\bullet$};
\draw node at (14,-1) {$\bullet$};
\draw node at (18,-1) {$\bullet$};

\draw [->, dotted] (10,5)--(8.2,4.1); 
\draw [->, dotted] (8,4)--(6.2,3.1); 
\draw [->, dotted] (12,4)--(10.2,3.1); 

\draw [->, dotted] (6,3)--(4.2,2.1); 
\draw [->, dotted] (10,3)--(8.2,2.1); 
\draw [->, dotted] (14,3)--(12.2,2.1); 

\draw [->, dotted] (4,2)--(2.2,1.1); 
\draw [->, dotted] (8,2)--(6.2,1.1); 
\draw [->, dotted] (12,2)--(10.2,1.1); 
\draw [->, dotted] (16,2)--(14.2,1.1); 

\draw [->] (10,5)--(11.8,4.1); 

\draw [->] (8,4)--(9.8,3.1); 
\draw [->] (12,4)--(13.8,3.1); 

\draw [->] (6,3)--(7.8,2.1); 
\draw [->] (10,3)--(11.8,2.1); 
\draw [->] (14,3)--(15.8,2.1); 

\draw [->] (4,2)--(5.8,1.1); 
\draw [->] (8,2)--(9.8,1.1); 
\draw [->] (12,2)--(13.8,1.1); 
\draw [->] (16,2)--(17.8,1.1); 

\draw [->, dashed] (2,1) --(2,-0.7);
\draw [->, dashed] (6,1) --(6,-0.7);
\draw [->, dashed] (10,1) --(10,-0.7);
\draw [->, dashed] (14,1) --(14,-0.7);
\draw [->, dashed] (18,1) --(18,-0.7);
\draw [->, dashed] (2,-1) --(2,-2.7);
\draw [->, dashed] (6,-1) --(6,-2.7);
\draw [->, dashed] (10,-1) --(10,-2.7);
\draw [->, dashed] (14,-1) --(14,-2.7);
\draw [->, dashed] (18,-1) --(18,-2.7);

\end{tikzpicture}
\end{equation}

\noindent where the solid and dashed arrows show the action of
$\e_i$, and the dotted or dashed arrows denote the action of $\e_i^*$.  Here the width of the diagram at the top is $\varepsilon_i(b_-)$, where $b_-$ is the bottom vertex. For any $b$ in this component such that $\varphi_i^*(b)=0$ (i.e. such that there are no dotted or dashed arrows pointing to $b$), $\varphi_i(b)$ is the longest path formed by solid arrows that ends at $b$, and  $\vareps_i(b)$ is the longest path formed by solid arrows that starts at $b$. 

The following operation $\sigma_i$ comes from \cite[Corollary 3.4.8]{Saito:1994}. We will refer to this as a Saito reflection. 

\begin{Definition} \label{def:ref} For $b \in B(-\infty)$ such that $\varphi_i^*(b)=0$ , define $\sigma_ib=
  ( \e_i^*)^{\epsilon_i(b)} \f_i^{\varphi_i(b)} b$.
\end{Definition}

As one would expect from the name,
$\sigma_i$ has the property that $\wt \sigma_i(b) = s_i \wt(b)$ where $s_i$ is the $i$th simple reflection. Note however that this is only true provided $\varphi_i^*(b)=0$.
There is also a notion of dual Saito reflection defined by $\sigma_i^*b = * \sigma_i *$ which acts as a reflections for those $b \in B(-\infty)$ such that $\varphi_i(b)=0$.
For our purposes, it is useful to have the following alternative
characterization of $\sigma_i$: 
   \begin{Proposition} \label{prop:ref} Fix $b \in B(-\infty)$. If $\varphi_i^*(b)=0$, then 
     $\sigma_i b=\f_i^{max}(\e_i^*)^N b$ for
     any $N \geq \vareps_i(b)$.
     Similarly, if $\varphi_i(b)=0$, then 
     $\sigma_i^* b=(\f_i^*)^{max}(\e_i)^N b$ for
     any $N \geq \vareps_i^*(b)$.

   \end{Proposition}
   \begin{proof}
Consider a component of $B(-\infty)$, as shown in \eqref{ii*-pic}. The nodes where $\varphi_i^*(b)=0$ are exactly those that have no dotted or dashed arrows pointing towards them, for instance the element $b$ in the diagram in that section. It is clear form the picture that the two formulas have the same effect on $b$. 
\end{proof}

\section{Rank 2 Affine MV polytopes}\label{sec:rank2}

\subsection{Rank 2 affine root systems} \label{sec:asl2polys}
The root sysems for $\asl_2$ and $A_2^{(2)}$ correspond to the affine Dynkin diagrams 

\begin{center}
\begin{tikzpicture}

\draw 
(-0.1, 0.2) --(0.7,0.2)
(-0.1, 0.1) --(0.7,0.1)

(-0.15,0.15) --(0,0.3)
(-0.15,0.15) --(0,0)

(0.75,0.15) --(0.6,0.3)
(0.75,0.15) --(0.6,0)
;

\draw node at (-0.4, 0.15) {$\bullet$};
\draw node at (1, 0.15) {$\bullet$};

\draw node at (-0.4, -0.35) {$0$};

\draw node at (1.15, -0.35) {$1 \;\;,$};

\draw node at (-1.5, 0) {$\asl_2$:};

\end{tikzpicture}
\qquad \qquad
\begin{tikzpicture}[xscale=-1]

\draw (0, 0.3) --(0.7,0.3)
(-0.1, 0.2) --(0.7,0.2)
(-0.1, 0.1) --(0.7,0.1)
(0, 0) --(0.7,0)

(-0.15,0.15) --(0.1,0.4)
(-0.15,0.15) --(0.1,-0.1)
;

\draw node at (-0.4, 0.15) {$\bullet$};
\draw node at (1, 0.15) {$\bullet$};

\draw node at (-0.6, -0.35) {$1  \;\;.$};

\draw node at (1, -0.35) {$0$};

\draw node at (2.5, 0) {$A_2^{(2)}$:};

\end{tikzpicture}
\end{center}
The corresponding symmetrized Cartan matrices are 
$$\asl_2: \quad N = 
\left(
\begin{array}{rr}
2 & -2 \\
-2 & 2
\end{array}
\right), \qquad
A_2^{(2)}: \quad N = 
\left(
\begin{array}{rr}
8 & -4 \\
-4 & 2
\end{array}
\right).
$$
Denote the simple roots by $\alpha_0, \alpha_1$. Define $\delta = \alpha_0+ \alpha_1$ for $\asl_2$ and $\delta =  \alpha_0+ 2\alpha_1$ for $A_2^{(2)}$. 
Note that we have chosen $\alpha_0$ to be the long root for $\ATT$, which is opposite from  the conventions in \cite{Kac90,BDKT:??}. We have instead followed the convention in \cite{Aka,BN} which is more convenient for the theory of affine PBW bases.

The type $\g$ weight space is a three dimensional vector space containing $\alpha_0, \alpha_1$. This has a standard non-degenerate bilinear form $(\cdot, \cdot)$ such that 
$(\alpha_i, \alpha_j) = N_{i,j}.$
Notice that $(\alpha_0, \delta)= (\alpha_1, \delta)=0$. Fix  fundamental coweights $\omega_0, \omega_1$ which satisfy $( \alpha_i, \omega_j) = \delta_{i,j}$, where we are identifying coweight space with weight space using $(\cdot, \cdot)$. 

The set of positive roots for $\asl_2$ is
\begin{equation} \label{eq:roots}
\{\alpha_0, \alpha_0+\delta, \alpha_0+2\delta, \ldots \} \sqcup \{\alpha_1, \alpha_1+\delta, \alpha_1+2\delta, \ldots \} \sqcup \{\delta, 2\delta, 3\delta \ldots \},
\end{equation}
where the first two families consist of real roots and the third family consists of imaginary roots. 
The set of positive roots for $A_2^{(2)}$ is 
\begin{equation}
 \{   \alpha_1+ k  \delta,   \alpha_0+2 k  \delta,   \alpha_1+  \alpha_0+k  \delta, 2 \alpha_1+ (2k+1) \delta \mid k \geq 0 \} \sqcup  \{ k \delta \mid k \geq 1 \},
\end{equation} 
where the first set consists of real roots and the second set of imaginary roots.
We draw these in the plane as
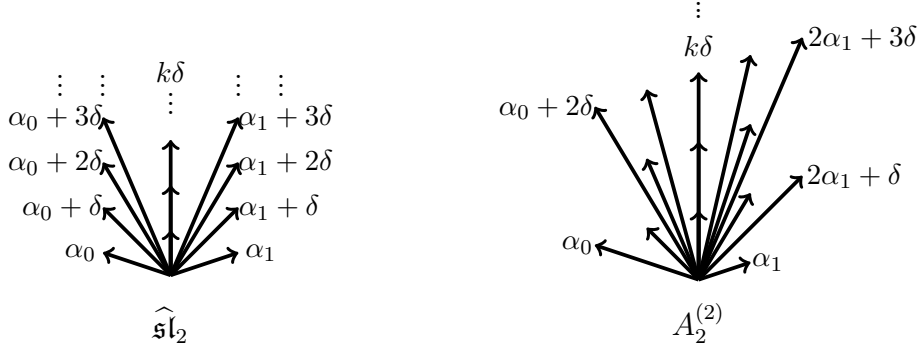
\begin{center}
\begin{tikzpicture}[scale=0.3]

\draw[line width = 0.05cm, ->] (0,0)--(3,1);
\draw[line width = 0.05cm, ->] (0,0)--(-3,1);

\draw[line width = 0.05cm, ->] (0,0)--(3,3);
\draw[line width = 0.05cm, ->] (0,0)--(-3,3);

\draw[line width = 0.05cm, ->] (0,0)--(3,5);
\draw[line width = 0.05cm, ->] (0,0)--(-3,5);

\draw[line width = 0.05cm, ->] (0,0)--(3,7);
\draw[line width = 0.05cm, ->] (0,0)--(-3, 7);

\draw[line width = 0.05cm, ->] (0,0)--(0,2);
\draw[line width = 0.05cm, ->] (0,0)--(0,4);
\draw[line width = 0.05cm, ->] (0,0)--(0,6);

\draw node at (0,8) {.};
\draw node at (0,7.6) {.};
\draw node at (0,7.2) {.};

\draw node at (-3,8.8) {.};
\draw node at (-3,8.4) {.};
\draw node at (-3,8) {.};

\draw node at (3,8.8) {.};
\draw node at (3,8.4) {.};
\draw node at (3,8) {.};

\draw node at (-4.9,8.8) {.};
\draw node at (-4.9,8.4) {.};
\draw node at (-4.9,8) {.};

\draw node at (4.9,8.8) {.};
\draw node at (4.9,8.4) {.};
\draw node at (4.9,8) {.};

\draw node at (-4,1) {{\small $ \alpha_0$}};
\draw node at (-4.8,3) {{\small $\alpha_0+\delta$}};
\draw node at (-5.1,5) {{\small $\alpha_0+2\delta$}};
\draw node at (-5.1,7) {{\small $\alpha_0+3\delta$}};

\draw node at (4,1) {{\small $\alpha_1$}};
\draw node at (4.8,3) {{\small $\alpha_1+\delta$}};
\draw node at (5.1,5) {{\small $\alpha_1+2\delta$}};
\draw node at (5.1,7) {{\small $\alpha_1+3\delta$}};

\draw node at (0,9) {{\small $k \delta$}};

\draw node at (0,-2.2) {$\asl_2$};

\end{tikzpicture}
$\qquad \qquad$
\begin{tikzpicture}[scale=0.23, xscale=-1]

\draw[line width = 0.05cm, ->] (0,0)--(-3,1);

\draw[line width = 0.05cm, ->] (0,0)--(6,2);
\draw[line width = 0.05cm, ->] (0,0)--(3,3);
\draw[line width = 0.05cm, ->] (0,0)--(0,4);
\draw[line width = 0.05cm, ->] (0,0)--(-3,5);
\draw[line width = 0.05cm, ->] (0,0)--(-6,6);

\draw[line width = 0.05cm, ->] (0,0)--(3,7);
\draw[line width = 0.05cm, ->] (0,0)--(0,8);
\draw[line width = 0.05cm, ->] (0,0)--(-3,9);

\draw[line width = 0.05cm, ->] (0,0)--(6,10);
\draw[line width = 0.05cm, ->] (0,0)--(3,11);
\draw[line width = 0.05cm, ->] (0,0)--(0,12);
\draw[line width = 0.05cm, ->] (0,0)--(-3,13);
\draw[line width = 0.05cm, ->] (0,0)--(-6,14);

\draw node at (0,16) {.};
\draw node at (0,15.6) {.};
\draw node at (0,15.2) {.};

\draw node at (-4,1) {{\small $ \alpha_1$}};
\draw node at (-9,6) {{\small $2 \alpha_1 +  \delta$}};
\draw node at (-9.5,14) {{\small $2 \alpha_1 + 3  \delta$}};

\draw node at (7,2) {{\small $ \alpha_0$}};
\draw node at (8.8,10) {{\small $ \alpha_0 + 2  \delta$}};

\draw node at (0,13.5) {{\small $k  \delta$}};

\draw node at (0,-2.5) {$A_2^{(2)}$};

\end{tikzpicture}
\end{center}

\subsection{Lusztig data and pseudo-Weyl polytopes}

\begin{Definition}
A {\bf Lusztig datum} $\cc$ is a choice of a non-negative integer $\cc_\beta$ for each positive real root $\beta$, all but finitely many of which are $0$, and a partition $\cc_\delta$. The {\bf weight}  of $\cc$ is $\wt(\cc) = |\cc_\delta| \cdot \delta+ \sum_{\beta \text{ real}} \cc_\beta \cdot \beta$.
\end{Definition}   

\begin{Remark}
Notice that the data of a partition $\cc_\delta$ is equivalent to the data of a number $\cc'_{k \delta}$ for each positive imaginary root $k \delta$, where $\cc'_{k \delta}$ is the number of parts of $\cc_\delta$ of size exactly $k.$ Thus a Lusztig datum is equivalent to a Kostant partition.  
\end{Remark}

If a Lusztig datum $\cc$ has $\cc_\beta=0$ for all positive real roots, we call $\cc$ {\bf purely imaginary}. We will sometimes abuse notation and write simply $\lambda$ to denote the purely imaginary Lusztig datum $\cc$ with $\cc_\delta=\lambda$.

\begin{Definition}{\label{DecoratedPseudoWeylPolytope}}
For either $\asl_2$ or $\ATT$, a {\bf decorated pseudo-Weyl polytope} $P$ consists of a pair of Lusztig data $\lbc(P)$ and $\rbc(P)$ of the same weight. This weight is called the weight of $P$. 
\end{Definition}

\begin{Remark}
What we are calling a decorated pseudo-Weyl polytope is called a {\bf decorated GGMS polytope} in \cite{BKT} and \cite{BDKT:??}. Our terminology originates from \cite{Kam}, where the notion of pseudo-Weyl polytope is defined in finite type.
\end{Remark}

In order to describe these geometrically (and justify the word polytope) we need some notation. 
Label the positive real roots by $\beta_k, \beta^k$ for $k \in \bz_{>0}$ as follows:
\begin{itemize}
\item For $\asl_2$:
$\beta_k = \alpha_1 + (k-1) \delta$ and $\beta^k = \alpha_0 + (k-1) \delta$. 

\item For $A_2^{(2)}$: 
 $$
 \beta_k=\begin{cases}
 \alpha_1+\frac{k-1}{2}\delta&\text{if $k$ is odd,}\\
 2 \alpha_1+(k-1)\delta&\text{if $k$ is even,}
 \end{cases}
 \qquad
 \beta^k=\begin{cases}
 \alpha_0+(k-1)\delta&\text{if $k$ is odd,}\\
 \alpha_0+\alpha_1+\frac{k-2}{2}\delta&\text{if $k$ is even.}
 \end{cases}
$$
\end{itemize}

To a decorated pseudo-Weyl polytope $P$, we associate an underlying polytope (up to translation) in the root lattice whose vertices $\{\mu^r_k,\mu^{r,k},\mu^\ell_k,\mu^{r,k}\}$ are defined by:
\begin{equation}
\begin{aligned}
& \mu^r_0 = \mu^\ell_0 \\
& \mu^r_k-\mu^r_{k-1} = \rbc_{\beta_k} \cdot \beta_k, \quad \mu^{r,k-1}-\mu^{r,k} = \rbc_{\beta^k} \cdot \beta^k \\
& \mu^\ell_k-\mu^\ell_{k-1} = \lbc_{\beta^k} \cdot \beta^k, \quad \mu^{\ell,k-1}-\mu^{\ell,k} = \lbc_{\beta_k} \cdot \beta_k. 
\end{aligned}
\end{equation}
Because Lusztig data take the value zero on all but finitely many roots, the vertices $\mu^r_k$ must all coincide for sufficiently large $k$, as must the vertices $\mu^{r,k},$ $\mu^\ell_k$, $\mu^{\ell,k}$. We denote 
\begin{equation}
\mu^r_\infty = \lim_{k \rightarrow \infty} \mu^r_k, \; \;\mu^{r,\infty} = \lim_{k \rightarrow \infty} \mu^{r,k}, \;\;
\mu^\ell_\infty = \lim_{k \rightarrow \infty} \mu^\ell_k, \;\;  \mu^{\ell,\infty} = \lim_{k \rightarrow \infty} \mu^{\ell,k}.
\end{equation}
See Figure \ref{MVpoly}.

\subsection{Definition and characterization of MV polytopes}

The following definition can be found in \cite{BDKT:??}, although we have changed some noation. 
\begin{Definition} \label{def:characterization} An {\bf MV polytope} is a decorated pseudo-Weyl polytope such that
\begin{enumerate}
\item \label{top-part} For each $k \geq 2$, $\max \{ ( \mu^\ell_{k}-\mu^r_{k-1}, \omega_1), (\mu^r_{k} -\mu^\ell_{k-1}, \omega_0) \} = 0$.

\item \label{bottom-part}  For each $k \geq 2$, $\min \{ (\mu^{\ell,k} -\mu^{r, k-1}, \omega_0), (\mu^{r,k} -\mu^{\ell, k-1}, \omega_1) \}= 0$.

\item \label{part:middle1} If the vectors $\mu^r_\infty -  \mu^\ell_\infty$ and $\mu^{r,\infty}- \mu^{\ell,\infty}$ are parallel, then $\rbc_\delta=\lbc_\delta$. Otherwise, one is obtained from the other by removing a part of size $ \frac{|\alpha_1|}{2|\alpha_0|}(\mu^r_\infty - \mu^\ell_\infty, \alpha_1)$. 

\item \label{part:middle2} $ (\rbc_\delta)_1, (\lbc_\delta)_1  \leq\frac{|\alpha_1|}{2|\alpha_0|}(\mu^r_\infty - \mu^\ell_\infty, \alpha_1),$ where e.g. $(\rbc_\delta)_1$ denotes the largest part of the partition $\rbc_\delta$. 
\end{enumerate}
See Figure \ref{MVpoly}. Let $\MV$ denote the set of MV polytopes.
\end{Definition}

\begin{figure}[ht]

\begin{center}
\begin{tikzpicture}[yscale=0.2, xscale=0.6]

\draw [line width = 0.01cm, color=gray] (-1.5,2.5) -- (3.5, 7.5);
\draw [line width = 0.01cm, color=gray] (-2,4) -- (4, 10);
\draw [line width = 0.01cm, color=gray] (-2.5,5.5) -- (4, 12);
\draw [line width = 0.01cm, color=gray] (-3.3333,8.6666) -- (4, 16);
\draw [line width = 0.01cm, color=gray] (-3.6666,10.3333) -- (4, 18);
\draw [line width = 0.01cm, color=gray] (-4.25,13.75) -- (4, 22);
\draw [line width = 0.01cm, color=gray] (-4.5,15.55) -- (4, 24);
\draw [line width = 0.01cm, color=gray] (-4.75,17.25) -- (4, 26);
\draw [line width = 0.01cm, color=gray] (-5,19) -- (4, 28);
\draw [line width = 0.01cm, color=gray] (-5,21) -- (4, 30);
\draw [line width = 0.01cm, color=gray] (-5,23) -- (4, 32);
\draw [line width = 0.01cm, color=gray] (-5,25) -- (4, 34);
\draw [line width = 0.01cm, color=gray] (-4.6666,29.3333) -- (3.6666, 37.6666);
\draw [line width = 0.01cm, color=gray] (-4.3333,31.6666) -- (3.33333, 39.3333);

\draw [line width = 0.01cm, color=gray] (1,1) -- (-2,4);
\draw [line width = 0.01cm, color=gray]  (2,2)-- (-3,7);
\draw [line width = 0.01cm, color=gray] (2.5,3.5) -- (-3.5,9.5);
\draw [line width = 0.01cm, color=gray] (3.3333,6.6666) -- (-4.3333,14.3333);
\draw [line width = 0.01cm, color=gray] (3.6666,8.3333) -- (-4.6666,16.6666);
\draw [line width = 0.01cm, color=gray] (4,12) -- (-5,21);
\draw [line width = 0.01cm, color=gray] (4,14) -- (-5,23);
\draw [line width = 0.01cm, color=gray] (4,16) -- (-5,25);
\draw [line width = 0.01cm, color=gray] (4,18) -- (-5,27);

\draw [line width = 0.01cm, color=gray] (4,20) -- (-4.75,28.75);
\draw [line width = 0.01cm, color=gray] (4,22) -- (-4.5,30.5);
\draw [line width = 0.01cm, color=gray] (4,24) -- (-4.25,32.25);
\draw [line width = 0.01cm, color=gray] (4,28) -- (-3.5,35.5);

\draw [line width = 0.01cm, color=gray] (4,32) -- (-2,38);
\draw [line width = 0.01cm, color=gray] (4,34) -- (-1,39);
\draw [line width = 0.01cm, color=gray] (4,36) -- (0,40);
\draw [line width = 0.01cm, color=gray] (3.5,38.5) -- (1,41);

\draw [line width = 0.01cm, color=gray] (-3,7) -- (4, 14);
\draw [line width = 0.01cm, color=gray] (-4,12) -- (4, 20);

\draw [line width = 0.01cm, color=gray] (-3,37) -- (4, 30);
\draw [line width = 0.01cm, color=gray] (-4,34) -- (4, 26);
\draw [line width = 0.01cm, color=gray] (-5,27) -- (4, 18);

\draw (0,0) node {$\bullet$};
\draw (-1,1) node {$\bullet$};
\draw (-3,7) node {$\bullet$};
\draw (-4,12) node  {$\bullet$};
\draw (-5,19) node  {$\bullet$};
\draw (-5,23) node {$\bullet$};
\draw (-5,25) node {$\bullet$};
\draw  (-5,27) node {$\bullet$};
\draw (-4,34) node {$\bullet$};
\draw (-3,37) node {$\bullet$};
\draw  (2,42)  node {$\bullet$};
\draw (2,2)  node {$\bullet$};
\draw (3,5) node {$\bullet$};
\draw (4,10) node {$\bullet$};
\draw (4,28) node {$\bullet$};
\draw (4,28) node {$\bullet$};
\draw (4,32) node {$\bullet$};
\draw (4,34) node {$\bullet$};
\draw (4,36) node {$\bullet$};
\draw (3,41) node {$\bullet$};

\draw [line width = 0.04cm] (0,0)--(-1,1);
\draw  [line width = 0.04cm] (-1,1)--(-3,7);
\draw [line width = 0.04cm] (-3,7)--(-4,12);
\draw [line width = 0.04cm] (-4,12)--(-5,19);

\draw [line width = 0.04cm] (-5,19)--(-5,27);
\draw [line width = 0.04cm] (-5,27)--(-4,34);
\draw [line width = 0.04cm] (-4,34)--(-3,37);
\draw [line width = 0.04cm](-3,37)--(2,42);
\draw [line width = 0.04cm](0,0)--(2,2);
\draw [line width = 0.04cm] (2,2)--(3,5);
\draw [line width = 0.04cm] (3,5)--(4,10);
\draw [line width = 0.04cm] (4,10)--(4,36);
\draw [line width = 0.04cm](4,36)--(3,41);
\draw [line width = 0.04cm] (3,41)--(2,42);

\draw [line width = 0.01cm, color=gray] (-1,1) -- (3,5);
\draw [line width = 0.01cm, color=gray] (3,5) -- (-4,12);
\draw [line width = 0.01cm, color=gray] (4,10) -- (-5,19);
\draw [line width = 0.01cm, color=gray] (-5,27) -- (4,36);
\draw [line width = 0.01cm, color=gray] (-4,34) -- (3,41);

\draw (1,1) node {\tiny $\bullet$};

\draw (1.2,-0.35) node {\tiny $\alpha_1$};
\draw (3,2.7) node {\tiny $\alpha_1+\delta$};
\draw (4.4,7) node {\tiny $\alpha_1+2\delta$};
\draw (4.6,23) node {\tiny $\delta$};
\draw (4.4,39) node {\tiny $\alpha_0+2\delta$};
\draw (2.8,42.5) node {\tiny $\alpha_0$};

\draw (-0.5,41) node {\tiny $\alpha_1$};
\draw (-4.15,36.3) node {\tiny $\alpha_1+\delta$};
\draw (-5.5,31) node {\tiny $\alpha_1+3\delta$};
\draw (-5.5,23) node {\tiny $\delta$};
\draw (-5.5,15) node {\tiny $\alpha_0+3\delta$};
\draw (-4.5,9) node {\tiny $\alpha_0+2\delta$};
\draw (-3,3.5) node {\tiny $\alpha_0+\delta$};
\draw (-0.7,-0.7) node {\tiny $\alpha_0$};

\draw(0,-1.2) node {$\mu_0$};
\draw(2.3,0.3) node {$\mu^r_1$};
\draw(3.5,4.5) node {$\mu^r_2$};
\draw(7.4,10) node {$\mu^r_3= \mu^r_4 = \cdots = \mu^r_\infty$};
\draw(7.7,36) node {$\mu^{r,3}= \mu^{r,4} = \cdots = \mu^{r,\infty}$};
\draw(4.9,41.7) node {$\mu^{r,1}= \mu^{r,2}$};
\draw(-1.4,0) node {$\mu^\ell_1$};
\draw(-3.5,6) node {$\mu^\ell_2$};
\draw(-4.5,12) node {$\mu^\ell_3$};
\draw(-8.4,19) node {$\mu^\ell_\infty = \cdots = \mu^\ell_5=  \mu^\ell_4$};
\draw(-8.7,27) node {$\mu^{\ell,\infty} = \cdots = \mu^{\ell,5}= \mu^{\ell,4}$};
\draw(-5.7,34) node {$\mu^{\ell,3}= \mu^{\ell,2}$};
\draw(-3.2,38.2) node {$\mu^{\ell,1}$};
\draw(2.1,43.7) node {$\mu^{0}$};

\draw[line width = 0.03cm, ->] (0,-4)--(1,-3);
\draw[line width = 0.03cm, ->] (0,-4)--(-1,-3);
\draw (1.5,-2.6) node {$\alpha_1$};
\draw (-1.5,-2.6) node {$\alpha_0$};

\end{tikzpicture}

\end{center}

\caption{\label{MVpoly}An $\asl_2$ MV polytope. The partitions labeling the vertical edges are indicated by including extra vertices on the vertical edges, such that the edge is cut into the pieces indicated by the partition. 
Here:
$$ \hspace{-3cm} \rbc_{\alpha_1}=2, \; \rbc_{\alpha_1+\delta}=1, \; \rbc_{\alpha_1+2\delta} =1,\; \rbc_\delta= (9,2,1,1), \; \rbc_{\alpha_0+2\delta}=1, \; \rbc_{\alpha_0}=1, $$
$$\hspace{-2.5cm} \lbc_{\alpha_0} =1,  \lbc_{\alpha_0+\delta} = 2,  \lbc_{\alpha_0+2\delta}=1,  \lbc_{\alpha_0+3\delta} = 1, \lbc_\delta =(2,1,1), \lbc_{\alpha_1+3\delta}=1, \lbc_{\alpha_1+\delta}=1,  \lbc_{\alpha_1}=5.$$
}

\end{figure}

\begin{Theorem} \cite{BDKT:??} \label{thm:lusztig-data}
For each Lusztig datum ${\bf c}$, there is a unique $P^r_{\bf c} \in \MV$ such that the right Lusztig data of $P^r_{\bf c}$ is given by ${\bf c}$. Similarly, there is a unique $P^\ell_{\bf c} \in \MV$ whose left Lusztig data is given by $\bf c$. \qed
\end{Theorem}

\begin{Definition} \label{def:crystal-ops}
Fix $P \in \MV$ with right and left Lusztig data ${\rbc}$ and ${\lbc}$ respectively. Then
$\e_0(P)$ is the MV polytope with right Lusztig data $\e_0({\rbc})$ and $\e_1(P)$ is the MV polytope with left Lusztig datum $\e_1({\lbc})$, where $\e_0(\rbc)$ agrees with $\rbc$ except that $\e_0({\rbc})_{\alpha_0}= \rbc_{\alpha_0}+1$, and $\e_1({\lbc})$ agrees with ${\lbc}$ except that $\e_1({\lbc})_{\alpha_1} = \lbc_{\alpha_1}+1$. 

Similarly, $\f_0(P)$ is the MV polytope with right data $\f_0({\rbc})$ and $\f_1(P)$ is the MV polytope with left Lusztig data $\f_1({\lbc})$, where $\f_0(\rbc)$ agrees with $\rbc$ except that $\f_0({\rbc})_{\alpha_0}= \rbc_{\alpha_0}-1$ and $\f_1({\lbc})$ agrees with ${\lbc}$ except that $\f_1({\lbc})_{\alpha_1} = \lbc_{\alpha_1}-1$; if $\rbc_{\alpha_0}$ or $\lbc_{\alpha_1}$ are zero then $\f_0$ or $\f_1$ sends that polytope to $\emptyset$.
\end{Definition}

\begin{Theorem} \cite{BDKT:??}
$\MV$ along with the operators $\e_0, \f_0, \e_1, \f_1$ realizes $B(-\infty)$.  \qed
\end{Theorem}

The following describes how Saito reflections act on MV polytopes. 

\begin{Proposition} \label{prop:sig0}
For any $b \in B(-\infty)$ with $\varphi_0(b)=0$,  we have $\lbc_{\alpha_0}(MV_{\sigma_0(b)})=0$, and for all other $\alpha$, $\lbc_{\alpha}(MV_{\sigma_0(b)}) = \rbc_{s_0\alpha}(MV_b)$. 
\end{Proposition}

\begin{proof}
Fix $b \in B(-\infty)$. 
By \cite[Remark 4.12]{BDKT:??}, for sufficiently large $N$ and all $k \geq 1$, $MV_{\tilde e_0^Nb}$ has 
$$(\mu^r_{k} -\mu^\ell_{k-1}, \omega_0) \quad \text{and} \quad (\mu^{\ell,k} -\mu^{r, k-1}, \omega_0)=0$$
(this is only stated there for $\asl_2$, but the same proof works for $A_2^{(2)}$). 
Hence each diagonal $(\mu^r_{k}, \mu^\ell_{k-1})$ and $(\mu^{\ell,k} ,\mu^{r, k-1})$ is parallel to $\alpha_1$, from which it is clear that, for all $\alpha \neq \alpha_0$, 
$$\cc^\ell_\alpha(MV_{\tilde e_0^Nb})= \cc^r_{s_0(\alpha)}(MV_{\tilde e_0^Nb})=\cc^r_{s_0(\alpha)}(MV_{b}).$$
The result is then immediate from Proposition \ref{prop:ref} (which says $\sigma_i(b) = (f_i^*)^{max} e_i^N(b)$ for large $N$). 
\end{proof}

\begin{Definition} \label{def:ref-data}
Given an Lusztig datum $\cc$ satisfying $\cc_{\alpha_i}=0$, we define a new Lusztig datum $\cc \circ s_i$ by $(\cc \circ s_i)_{\alpha_i}=0$, $(\cc \circ s_i)_\beta = \cc_{s_i(\beta)}$ and $(\cc \circ s_i)_\delta = \cc_0$.
For $\asl_2$, 
we define $\cc \circ \tau$ by $(\cc \circ \tau)_{\alpha_i + \delta} = \cc_{\alpha_{i+1} +\delta}$ and $(\cc \circ \tau)_\delta = \cc_\delta$. 

\end{Definition}

\begin{Theorem} \label{th:unique-aff} Assume $\g$ is of type $\asl_2$ or $A_2^{(2)}$. There is
  a unique map $b \rightarrow P_b$ from $B(-\infty)$ to type $\g$
  decorated pseudo-Weyl polytopes (considered up to translation) such
  that, for all $b \in B(-\infty)$ and $i=0$ or $1$, the following hold. 
  
    \begin{enumerate}
    \item[(W)] \label{CC1-aff} $\wt(b)= \wt(P_b)$. 

    \item[(C1)] \label{CC2-aff} $\rbc_{\alpha_0}(P_{\e_0 b}) =
      \rbc_{\alpha_0}(P_b)+1$, and for all $\alpha \neq \alpha_0$,
      $\rbc_\alpha(P_{\e_0 b})= \rbc_\alpha(P_b)$. 
     \item[(C2)]  $\lbc_{\alpha_1}(P_{\e_1 b}) = \lbc_{\alpha_1}(P_b)+1$, and  for all $\alpha \neq \alpha_1$, $\lbc_\alpha(P_{\e_1 b})= \lbc_\alpha(P_b)$.
              \item[(C3)] $\lbc_{\alpha_0}(P_{\e_0^*b}) = \lbc_{\alpha_0}(P_b)+1$, and for all $\alpha \neq \alpha_0$,
         $\lbc_\alpha(P_{\e_0^*b})=
         \lbc_\alpha(P_b)$.
     \item[(C4)]  $\rbc_{\alpha_1}(P_{\e_1^*b}) =
        \rbc_{\alpha_1}(P_b)+1$, and  for all $\alpha \neq \alpha_1$,
         $\rbc_\alpha(P_{\e_1^*b})=\rbc_\alpha(P_b)$.

    \item[(S1)] \label{CC4-aff} 
    If $\varphi_0(b)=0$, then 
    $\lbc(P_{\sigma_0(b)})= \rbc(P_b) \circ s_0$.
      \item[(S2)]
	If $\varphi_1(b)=0$, then 
      $\rbc(P_{\sigma_1(b)})= \lbc(P_b) \circ s_1$.
      \item[(S3)]
     If $\varphi^*_0(b)=0$, then 
       $\rbc(P_{\sigma_0^*(b)})=  \lbc(P_b) \circ s_0$.
      \item[(S4)]
	 If $\varphi^*_1(b)=0$, then 
      $\lbc(P_{\sigma^*_1(b)})=\rbc(P_b) \circ s_1$.
      
    \item[(I)] \label{CC5-aff} If $\lbc_\beta(P_b)=0$ for all real roots $\beta$, and $\lbc_\delta(P_b)= \lambda \neq 0$, then $\rbc_{\alpha_1}(P_b)= \frac{|\alpha_0|}{|\alpha_1|} \lambda_1$, $\rbc_\delta(P_{b})=\lambda \backslash \lambda_1$; $\rbc_{\alpha_0}(P_{b}) =\lambda_1$; and $\rbc_\beta(P_b)=0$ for all other $\beta \in \tilde \Delta_+$.     \end{enumerate}
   This map takes each $b \in B(-\infty)$ to its MV polytope $MV_b$ as defined in \S\ref{sec:asl2polys}. 

 \end{Theorem}

\begin{Remark} \label{rem:K} One can easily see that 
Theorem \ref{th:unique-aff} remains true if (C3), (C4), (S3) and (S4) are replaced with the single condition
\begin{enumerate}
\item[(K)] for all $b$, $P_{b^*}= - P_b$ (up to translation) where $*$ is Kashiwara's involution. 
 \end{enumerate}
This gives a shorter statement, but the version given here is a-priori stronger. 
\end{Remark}

\begin{Remark}
A simplified version of Theorem \ref{th:unique-aff} also holds in finite-type rank-2 cases, where one only needs conditions (W), (C1), (C2), (S1) and (S2). This was previously observed in \cite{BK??} (see the discussion just before Remark 27). The proof is also contained in the proof of Theorem \ref{th:unique-aff} below. 
\end{Remark}

\begin{Remark}
The combinatorics from \cite{BDKT:??} is only used in the existence part of the proof below. One could instead prove existence using the construction of these polytopes describe in \S\ref{sec:pbw} or \S\ref{sec:quiver}, so our characterization does not fundamentally rely on \cite{BDKT:??}.
\end{Remark}

\begin{proof}[Proof of Theorem \ref{th:unique-aff}] 
We give the details of the proof only for the case $\asl_2$; the case of $A_2^{(2)}$ proceeds by the same argument, but the notation gets messier. 

We first show that the map $b \rightarrow MV_b$ has all the required properties. Properties (W) and (C1)-(C4) are immediate from the definitions of the crystal operators in \S\ref{sec:asl2polys}. Property (S1) is Proposition \ref{prop:sig0}, and (S2)-(S4) follow by symmetric arguments. 

To see (I), notice that 
\begin{equation} \label{eq:trap}
\begin{tikzpicture}[xscale=0.6, yscale=-0.2]
\draw

(0,0) node {$\bullet$}
(2,2) node {$\bullet$}
(2,8) node {$\bullet$}
(0,10) node {$\bullet$};

\draw[line width = 0.05cm] 
(0,0)--(2,2)--(2,8)--(0,10)--cycle;

\draw
(1.8,-0.6) node {$\rbc_{\alpha_0} = \lambda_1$}
(3.8,5) node {$\rbc_\delta= \lambda \backslash \lambda_1$}
(-1.6,5) node {$\lbc_\delta=\lambda$}
(1.8,10.8) node {$\rbc_{\alpha_1}=\lambda_1$};

\end{tikzpicture}
\end{equation}
is an MV polytopes according to Definition \ref{def:characterization}. Since by Theorem \ref{thm:lusztig-data} there is exactly one MV polytope with each right (or left) Lusztig datum, (I) follows.

Now suppose we have a map $b \rightarrow P_b$ that satisfies the conditions above. 
It suffices to show that $\rbc(P_b) = \rbc(MV_b)$ and $\lbc(P_b) = \lbc(MV_b)$ for all $b$.
We proceed by induction, the base case where $b$ is the lowest weight element being trivial by (W). 
Consider the partial order on Lusztig data where ${\ba} < \tilde \ba$ if 
\begin{itemize}
\item $|\ba_\delta| < |\tilde\ba_\delta|$, or
\item $|\ba_\delta| = |\tilde \ba_\delta|$ and $\# \{ \alpha \mid \ba_\alpha \neq 0 \} < \# \{ \alpha \mid \tilde \ba_\alpha \neq 0 \}.$
\end{itemize}
Fix a Lusztig datum $\cc$, and make the inductive hypothesis that  $\rbc(P_b) = \rbc(MV_b)$ whenever $\rbc (P_b) < \cc$ and that $\lbc(P_b) = \lbc(MV_b)$ whenever $\lbc(P_b)< \cc$. We will show that, if an element $b_\cc \in B(-\infty)$ has $\lbc(P_{b_\cc})=\cc$, then $\lbc(MV_{b_\cc})=\cc$ as well.

First assume that $\ba_\alpha \neq 0$ for some real root $\alpha$. Find $k$ minimal such that $\ba_{\alpha_0+k \delta}$ or $\ba_{\alpha_1+k \delta}$ is non-zero. We proceed in the case  $\ba_{\alpha_1+k \delta} \neq 0$, the other case following by a similar argument (using $*$ reflections and operators instead of the unstarred ones). Let
$$b' = \f_{k+1}^{\ba_{\alpha_1+k \delta}} \sigma_k \cdots \sigma_1  b_{\bf c},$$
where we read subscripts modulo $2$. If $k$ is even, it follows from axioms (S1) and (S2) and (C2)  that 
$$\lbc(P_{b_{\bf c}})=\lbc(MV_{b_{\bf c}})  \text{ if and only if } \lbc(P_{b'})= \lbc(MV_{b'}).$$
But $\lbc(P_{b'}) < \cc$, so by induction $\lbc(P_{b'})=\lbc(MV_{b'})$, hence $\cc = \lbc(P_{b_\cc})=\lbc(MV_{b_\cc})$.

If $k$ is odd, it follows from axioms (S1) and (S2) and (C1)  that 
$$\lbc(P_{b_{\bf c}})=\lbc(MV_{b_{\bf c}}) \text{ if and only if } \rbc(P_{b'})=\rbc(MV_{b'}).$$
Again $\rbc(P_{b'}) < \ba$ so by induction $\rbc(P_{b'})=\rbc(MV_{b'})$, hence $\cc =\lbc(P_{b_{\bf c}})=\lbc(MV_{b_{\bf c}})$. 

Now assume $\cc_\alpha = 0$ for every real root $\alpha$. Let $\cc'$ be the Lusztig data defined by $\cc'_\delta = \cc_\delta \backslash (\cc_\delta)_1$, $\cc'_{\alpha_0}=\cc'_{\alpha_1}= (\cc_\delta)_1$, and $\cc_\beta=0$ for all other real roots $\beta$.
By (I), $\rbc(P_{b_\cc}) = \cc'$, and by induction $\rbc(MV_{b_\cc}) = \cc'$. Since we already know that \eqref{eq:trap} is the unique MV polytope such that $\rbc(MV_{b_\cc}) = \cc'$, this implies $\lbc(MV_{b_\cc}) = \cc$.

To complete the argument, we must also show that for any element $b_{\bf c}'$ such that $\rbc(P_{b_{\bf c}'})= {\cc}$ we also have $\rbc(MV_{b_{\bf c}'})= {\bf c}$. This proceeds by an identical argument, with the only subtlety being that (I) is not symmetric. However, if $\ba_\alpha = 0$ for every real root $\alpha$, then we proceed by noticing that, by (S1), $\rbc(P_{b_{\bf c}'})=\rbc(MV_{b_{\bf c}'})$ if and only if $\lbc(MV_{\sigma_0 {b_{\bf c}'}}) = \lbc(P_{\sigma_0{b_{\bf c}'}})$, so the previous argument applies.
\end{proof}

The following modification of Theorem \ref{th:unique-aff} will be needed when we consider the pseudo-Weyl polytopes arising from PBW bases. It is immediate from the proof above.

\begin{Proposition} \label{rem:partial-char} 
 Fix a partition $\lambda$. Assume all the conditions of the Theorem \ref{th:unique-aff} hold except that (I) is only known to hold for all partitions $\mu$ with $|\mu|< |\lambda|$ or $\mu=\lambda$.  Then the $ \lbc(MV_b) = \lbc(P_b)$  for all $b$ with $|\lbc_\delta(P_b)| <\lambda$ or $\lbc_\delta(P_b)= \lambda$. Similarly $ \rbc(MV_b) = \rbc(P_b)$ for all $b$ with $|\rbc_\delta(P_b)| <\lambda$ or $\rbc_\delta(P_b) =\lambda$. \qed
\end{Proposition}

\section{Rank-2 affine MV polytopes from PBW bases} \label{sec:pbw}

\subsection{Definition and basic properties of rank-2 affine PBW bases}
For $\asl_2$ and $\ATT$, we consider the following two PBW bases coming from the work of Beck-Chari-Pressley, Akasaka, and Beck-Nakajima \cite{BCP,Aka,BN}. 
Following Beck-Nakajima, for each Lusztig datum $\cc$ define two elements of $\U^+$ by
\begin{align}
{\label{LCZero}} &
L(\cc,0) := E_1^{(\cc_{\alpha_1})} T_1^{-1}(E_0)^{(\cc_{s_1(\alpha_0)})} \cdots S_{\cc_\delta} \cdots  T_0(E_1)^{(\cc_{s_0(\alpha_1)})} E_0^{(\cc_{\alpha_0})}, \\
{\label{LCOne}} &
L(\cc,1) := E_0^{(\cc_{\alpha_0})} T_0^{-1}(E_1)^{(\cc_{s_0(\alpha_1)})} \cdots T_0^{-1}(S_{\cc_\delta}) \cdots T_1(E_0)^{(\cc_{s_1(\alpha_0)})} E_1^{(\cc_{\alpha_1})}.
\end{align} 
All the notation here except $S_{\cc_\delta}$ is defined in \S\ref{ss:qg}. We do not need to define $S_\lambda$ exactly; it will be enough to recall that it is a polynomial in the commuting variables $\tpsi_k$ for $k \geq 1$, and that for all $\lambda, \mu$, $S_\lambda S_\mu = \sum_\nu a_{\lambda, \mu}^\nu S_\nu$ where $a_{\lambda, \mu}^\nu$ are the Littlewood-Richardson coefficients. For $\asl_2$, the vectors $\tpsi_k$ (which is denoted $\tpsi_{k,1}$ in \cite{BCP}) are defined by 
\begin{equation} \label{eq:aslpsi}
\tpsi_k:= E_{k \delta - \alpha_1} E_{\alpha_1} - q^{-2} E_{\alpha_1} E_{k \delta - \alpha_1},
\end{equation}
where by definition $E_{k\delta - \alpha_1} = T_0 T_1  \cdots T_{k-2}(E_{k-1})$ (here subscripts are taken modulo $2$). For $\ATT$, $\tpsi_k$ are defined in \cite[Definition 3.3]{Aka} by
\begin{equation} \label{eq:attpsi}
\tpsi_k:= E_{\delta - \alpha_1} E_{(k-1)\delta + \alpha_1} - q^{-1} E_{(k-1)\delta +\alpha_1} E_{\delta - \alpha_1},
\end{equation}
where by definition $E_{\delta - \alpha_1} = T_0(E_1)$ and $E_{(k-1)\delta+\alpha_1}= (T_1^{-1} T_0^{-1})^{k-1}(E_1)$.

\begin{Remark}
Our notation regarding Lusztig data is slightly different from that of Beck-Nakajima. For them, $\cc$ is a sequence of non-negative integers indexed by the integers, and what we call $\cc_\delta$ is denoted $\cc_0$. We have translated their results into our notation.
\end{Remark}

It is well known that $\cB$ is a crystal basis of $\cL/q^{-1}\cL$, and it is shown in \cite{BCP,Aka,BN} that $\{ L(\cc,i) + q^{-1} \cL : \cc \text{ is a Lusztig datum} \} $ is also a crystal basis of $\cL/q^{-1} \cL$. By the uniqueness of crystal basis (see e.g. \cite[Theorem 8.1]{Kashiwara:1995}), we see that:

\begin{Theorem}{\label{PBWParameterizationOfCanonicalBasis}}
For each $i$, $\{ L(\cc,i) + q^{-1} \cL : \cc \text{ is a Lusztig datum } \} = \cB+ q^{-1} \cL .$ \qed
\end{Theorem}

\noindent Thus for each $i$ the PBW basis elements index $\cB$. Using this indexing some of the crystal operators are given by explicit formulas:

\begin{Proposition}{\cite[Formula 5.3]{BN}}{\label{ExplicitPBWCrystalFormula}}

\begin{enumerate}
  \item $\e_1 L(\cc,0) = L(\e_1 \cc, 0)$
  \item $\e^*_0 L(\cc,0) = L(\e_0 \cc, 0)$
  \item $\e_0 L( \cc, 1) = L( \e_0 \cc, 1)$
  \item $\e^*_1 L( \cc, 1) = L( \e_1 \cc, 1)$
\end{enumerate}
Here $\e_i \cc$ is the Lusztig datum with $(\e_i \cc)_ {\alpha_i} = \cc_ {\alpha_i} +1$ and otherwise agreeing with $\cc$. \qed
\end{Proposition}

Beck-Nakajima actually define a whole family of PBW bases $\{ L(\cdot,i)\}$, one for each $i \in \Z$. For $i \geq 1$ the basis vectors are given by 
$$
L(\cc,i) := E_{i-1}^{(\cc_{\alpha_{i-1}} )} T_{i-1}^{-1}(E_{i-2})^{(\cc_{ s_{i-1} ( \alpha_{i-2} ) } )} \cdots T_{i-1}^{-1} \cdots 
T_0^{-1}(S_{\cc_\delta}) \cdots T_{i}(E_{i+1})^{(\cc_{ s_{i}( \alpha_{i+1})})} E_{i}^{(\cc_{\alpha_{i}})},
$$
where all subscripts are taken modulo 2.
There is similar formula for $i \leq 0$. However, in the rank-2 cases we are considering, these various bases all coincide with either the case $i=0$ or $i=1$ (see Proposition \ref{prop:bases-coincide} below).

\begin{Lemma} \label{lem:*psi}
$\tpsi_k^* = T_1 \tpsi_k$. In type $\asl_2$ this is also equal to $\tau  \tpsi_k$. In particular, for any partition $\lambda$, 
$S_\lambda^* = T_1 S_\lambda$ and in type $\asl_2$ this is also equal to $\tau S_\lambda$.
\end{Lemma}

\begin{proof}
For type $A_2^{(2)}$ this is \cite[Proposition 3.26 (ii)]{Aka}. For $\asl_2$, the statement $T_1(\tpsi_k) =\tau \tpsi_k$ comes from \cite[Proposition 2]{Beck}. To finish it suffices to show that, in type $\asl_2$, $\tpsi_k^* = \tau \tpsi_k$. 

Since $*$ is an algebra anti-automorphism one can easily calculate from \eqref{eq:aslpsi} that $\psi_{1,1}^* = \tau \psi_{1,1}$. For $k>0$, \cite[Proposition 1.2]{BCP} gives the following:
\begin{align}
\label{eq:*t} \tpsi_{l+m} = E_{l\delta - \alpha_1}E_{m\delta + \alpha_1} - q^{-2}E_{m\delta + \alpha_1}E_{l\delta - \alpha_1} \text{ for } l >0, m \geq 0.
\end{align}
Here $E_{l\delta - \alpha_1} = T_0 T_1  \cdots T_{l-2}(E_{l-1})$ and $E_{m\delta + \alpha_1} = T_1^{-1} T_2^{-1} \cdots T_m^{-1}(E_{m+1})$.

Using the relation $T_i \circ * = * \circ T^{-1}_i$, it follows that $E_{l\delta - \alpha_1}^* = \tau E_{(l-1)\delta + \alpha_1}$ and $E_{m\delta+\alpha_1}^* = \tau E_{(m+1)\delta-\alpha_1}$. Replacing the pair $(l,m)$ by $(m+1,l-1)$ in \eqref{eq:*t} gives $\tpsi_k^* = \tau \tpsi_k$.
\end{proof}

\begin{Lemma}{\label{BraidOpImagLemma}}
For all $k$, $T_0T_1(\tpsi_k) =  \tpsi_k$. In particular, for any $\lambda$, 
$T_0T_1(S_{\lambda}) = S_{\lambda}$. 
\end{Lemma}

\begin{proof}
It type $\asl_2$, this follows from Lemma \ref{lem:*psi} since $T_0T_1(\tpsi_k) =\tau T_0 \tau T_1 (\tpsi_k)$. 
For $A_2^{(2)}$ this is \cite[Proposition 3.26 (i)]{Aka}.
\end{proof}

\begin{Proposition}{\label{prop:bases-coincide}} Let $i,j \in \bz$ be congruent modulo $2$, and let $\cc$ be any Lusztig datum. We then have $L(\cc,i) = L(\cc,j)$.
\end{Proposition}
\begin{proof}
Suppose both $i,j \geq 1$. By the definition of $L(\cc,i)$, it suffices to prove that  $$T_{i-1}^{-1} \cdots T_1^{-1}T_0^{-1}(S_{\cc_\delta}) = T_{j-1}^{-1} \cdots T_1^{-1}T_0^{-1}(S_{\cc_\delta}).$$ But this follows immediately from Lemma \ref{BraidOpImagLemma}. The other cases are similar.
\end{proof}

\begin{Proposition} \label{prop:tL} \label{PBWStarProposition} Fix $i=0$ or $1$ and let $\cc$ be a Lusztig datum. Then $L(\cc,i)^* = L(\cc,i-1)$, and $\tau L(\cc,i) = L(\cc \circ \tau, i-1)$. Furthermore if $\cc_{\alpha_{i}} = 0$ then $L(\cc \circ s_i,i) = T_i L(\cc,i-1)$.
\end{Proposition}

\begin{proof}
First assume $\cc_{\alpha_{0}} = 0$.
\begin{equation*}
\begin{aligned}
T_1 L(\cc,0) & = T_1 \left( E_1^{(\cc_{\alpha_1})} T_1^{-1}(E_0)^{(\cc_{s_1(\alpha_0)})} \cdots S_{\cc_\delta} \cdots  T_0(E_1)^{(\cc_{s_0(\alpha_1))}} E_0^{(\cc_{\alpha_0})} \right) \\
& = (E_0)^{(\cc_{s_1(\alpha_0)})} T_0^{-1} (E_1)^{(\cc_{s_1 s_0(\alpha_1)})}  \cdots T_1 S_{\cc_\delta} \cdots  T_1T_0(E_1)^{(\cc_{s_0(\alpha_1))}} T_1E_0^{(\cc_{\alpha_0})}   \\
& = (E_0)^{(\cc_{s_1(\alpha_0)})} T_0^{-1} (E_1)^{(\cc_{s_1 s_0(\alpha_1)})}  \cdots T_0^{-1}T_0 T_1 S_{\cc_\delta} \cdots  T_1T_0(E_1)^{(\cc_{s_0(\alpha_1))}} T_1E_0^{(\cc_{\alpha_0})}   \\
& = (E_0)^{(\cc_{s_1(\alpha_0)})} T_0^{-1} (E_1)^{(\cc_{s_1 s_0(\alpha_1)})}  \cdots T_0^{-1} S_{\cc_\delta} \cdots  T_1T_0(E_1)^{(\cc_{s_0(\alpha_1))}} T_1E_0^{(\cc_{\alpha_0})},
\end{aligned}
\end{equation*}
where the last equality is from Lemma \ref{BraidOpImagLemma}.

The statement that if $\cc_{\alpha_{0}} = 0$ then $L(\cc \circ s_i,1) = T_0 L(\cc,0)$ follows by a similar (and slightly shorter) argument. The statements that $L(\cc,i)^* = L(\cc,i-1)$ and $\tau L(\cc,i) = L(\cc \circ \tau, i-1)$ follow by the same logic, but using Lemma \ref{lem:*psi} in place of Lemma \ref{BraidOpImagLemma}, and making use of \cite[Corollary 1.3.3]{Saito:1994} which says that $T_i \circ * = * \circ T^{-1}_i$.
\end{proof}

\subsection{Relationship of PBW bases with the canonical basis}

\begin{Definition}{\label{BCIDefinition}}
For each $i$ and Lusztig datum $\cc$, denote by $b(\cc,i)$ the unique element of $\cB$  that coincides with $L(\cc,i)$ in $\cL/q^{-1} \cL$. 
\end{Definition}

\begin{Definition}
The partial order $\prec_0$ on Lusztig data is defined as follows: For any Lusztig datum $\cc$, we form two infinite tuples
\begin{align*}
\cc_{+_0} = (\cc_{\alpha_1}, \cc_{s_1(\alpha_0)}, \cdots ) \text{ and } \cc_{-_0} =( \cc_{\alpha_0}, \cc_{s_0(\alpha_1)}, \cdots ) 
\end{align*}
\noindent We say $\cc \prec_0 \cc^\prime$ if 
\begin{align*}
\cc_{+_0} \leq \cc_{+_0}^\prime \text{ and } \cc_{-_0} \leq \cc_{-_0}^\prime
\end{align*}

\noindent with one of these inequalities strict. Here, $\leq$ is the left-to-right lexicographic order. 

The partial  order $\prec_1$ is defined in the same way by swapping the roles of $0$ and $1$. 
\end{Definition}

\begin{Lemma}{\cite[Lemma 3.30]{BN}}{\label{ReorderingLemma}}
Fix $i$ modulo $2$. Let $\cc$ and $\cc^\prime$ be Lusztig data. Write:

\begin{align*}
L(\cc,i) L(\cc^\prime,i) = \sum_{\cc^{\prime\prime}} a^{\cc^{\prime\prime}}_{\cc,\cc^\prime} L(\cc^{\prime\prime},i)
\end{align*}
Then every Lusztig datum $\cc^{\prime\prime}$ that shows up in the righthand side sum satisfies:
\begin{align*}
\cc^{\prime\prime}_{+_i} \geq \cc_{+_i} \text{ and } \cc^{\prime\prime}_{-_i} \geq \cc^{\prime}_{-_i}.
\end{align*}

\vspace{-0.8cm}

\qed
\end{Lemma}

\begin{Theorem}{\cite[Theorem 3.13]{BN}}{\label{UpperTriangularityTheorem}}
The change of basis matrix between $\{ L(\cc,i) : \cc \text{ is a Lusztig datum } \}$ and $\{ b(\cc,i) : \cc \text{ is a Lusztig datum } \}$ is upper triangular with $1$'s on the diagonal with respect to the partial order $\prec_i$ on Lusztig data defined above.\qed
\end{Theorem}

The following is immediate from certain geometric constructions of the canonical basis and $B(-\infty)$, but we give an algebraic proof for completeness. 

\begin{Proposition}{\label{CanonicalBasisCrystalFormula}}
Fix $b \in \cB$. Write $E^{(n)}_i b = \sum_{b^\prime} a_{b,b^\prime} b^\prime$. Then $a_{b,\e^n_i b} \neq 0$.
\end{Proposition}

\begin{proof}
For notational clarity we present the proof for $i=1$; the case $i=0$ uses exactly the same argument. There is a unique Lusztig datum $\cc$ such that $b = b(\cc,0)$.  By Theorem \ref{UpperTriangularityTheorem}
\begin{align}
b(\cc,0) = L(\cc,0) + \sum_{\cc^\prime \succ_0 \cc} a_{\cc,\cc^\prime} L(\cc^\prime,0).
\end{align}
Multiplying both sides by $E_1$, 
\begin{align}
E^{(n)}_1b(\cc,0) = E^{(n)}_1L(\cc,0) + \sum_{\cc^\prime \succ_0 \cc} a_{\cc,\cc^\prime} E^{(n)}_1L(\cc^\prime,0).
\end{align}
But we know that up to scaling by a quantum integer $E^{(n)}_1L(\dd,0) = \e^{(n)}_1 L(\dd,0)$. So after scaling, each term on the right hand side lies in the PBW basis. Rewrite each term using the canonical basis. Again using Theorem \ref{UpperTriangularityTheorem} only $E^{(n)}_1L(\cc,0)$ will contribute the the coefficient of $\e^{(n)}_1 b(\cc,0)$. Thus there is no cancellation and the Proposition holds.
\end{proof}

\begin{Theorem}{\cite[Proposition 3.4.7]{Saito:1994} }{\label{SaitoReflectionThm}}
Suppose $b$ is an element of the canonical basis such that $T_i(b) \in \U^+$. Then $T_i(b)$ lies in the crystal lattice, and
\begin{align}
T_i(b) \equiv \sigma_i b \; \text{ in } \cL/q^{-1} \cL.
\end{align}

\vspace{-0.8cm}

\qed
\end{Theorem}

Using this formula and our explicit description of braid operators on PBW basis vectors, we have the following corollary to Saito's theorem. A version of this formula is mentioned in \cite[Remark 3.29]{BN}.

\begin{Corollary}{\label{SaitoReflectionFormula}}
If $\cc$ is a Lusztig datum with $\cc_{\alpha_i} = 0$ then $\sigma_i b(\cc,i-1) \in \cL$, and
\begin{align}
b(\cc \circ s_i, i) \equiv \sigma_i b(\cc,i-1) \; \text{ in } \cL/q^{-1} \cL.
\end{align}
\end{Corollary}

\begin{proof}

Write $b(\cc,i-1) = \sum_{\bf d} a_\cc^\dd L(\dd, i-1)$. By Theorem \ref{UpperTriangularityTheorem}, every $\dd$ with $a_\cc^\dd \neq 0$  must have $\dd_{\alpha_i}=0$. In particular, $T_i L(\dd, i-1) \in \U^+$. This implies $T_i b(\cc,i-1) \in \U^+$, so we can apply Theorem \ref{SaitoReflectionThm} to conclude that $T_i b(\cc,i-1) = \sigma_ib(\cc,i-1)$ in $\cL/ q^{-1} \cL$. 

By Proposition \ref{prop:tL}, $L(\cc \circ s_i,i) = T_i L(\cc,i-1)$, and by definition $L(\cc \circ s_i,i) \equiv b(\cc \circ s_i,i)$ and $L(\cc,i-1) \equiv b(\cc,i-1)$ in $\cL/q^{-1} \cL$.
\end{proof}

\begin{Lemma}{\cite[Lemma 4.1]{BCP}\cite[Theorem 8.5,  c.f. Proof of Theorem 8.17 ]{Aka}}{\label{TriangleLemma}}

Fix a positive integer $n$. 
\begin{enumerate}

\item For $\asl_2$, let $\cc$ be the Lusztig datum satisfying $\cc(\alpha_0) = \cc(\alpha_1) = n$, and let $(n)$ denote the partition consisting of one part of length $n$.

\item For $\ATT$, let $\cc$ be the Lusztig datum satisfying $\cc(\alpha_0) = n,  \cc(\alpha_1) = 2n$, and let $(n)$ denote the partition consisting of one part of length $n$.

\end{enumerate}
Then for each $i$, $b(\cc,i) = b((n),i+1)$. \qed
\end{Lemma}

\begin{Remark} We have not stated the most general versions of the results in this section. With appropriate definitions Lemma \ref{ReorderingLemma}, Theorem \ref{UpperTriangularityTheorem}, and Proposition \ref{CanonicalBasisCrystalFormula} hold for all affine Kac-Moody algebras (c.f the discussion immediatedly following \cite[Theorem 3.13]{BN}). Theorem \ref{SaitoReflectionThm} holds for all symmetrizable Kac-Moody algebras.
\end{Remark}

\subsection{Relationship with MV Polytopes}

Recall from Definition \ref{BCIDefinition} that, since $\cB$ and the two PBW bases agree as crystal bases, we can parameterize $\cB$ by Lusztig data. We will use the notation $b = b(\cc^r_b,1) = b(\cc^\ell_b,0)$ to denote the Lusztig data corresponding to $b$ with respect to the two PBW bases (and will drop the subscripts of $b$ on $\cc^r$, $\cc^\ell$ where it will not cause confusion). Thus by Definition \ref{DecoratedPseudoWeylPolytope} each $b \in \cB$ defines a decorated pseudo-Weyl polytope $\text{PBW}_b$ whose left Lusztig datum is $\cc^\ell_b$ and whose right Lusztig datum is $\cc^r_b$. Equivalently, this gives a map $b \mapsto \text{PBW}_b$ from $B(-\infty)$ to decorated pseudo-Weyl polytopes.

\begin{Theorem}{\label{PBWPolytopeTheorem}} For each $b \in B(-\infty)$, $PBW_b$ coincides with the affine MV polytope $MV_b$ as defined in \cite{BDKT:??} (and in \S\ref{sec:asl2polys}). 
\end{Theorem}

\begin{proof} 

It suffices to show that the map $b \rightarrow PBW_b$ satisfies the conditions of Theorem \ref{th:unique-aff}. Condition (W) is immediate since $\wt(\cc) = \wt( b(\cc,i) )$. Conditions (C1)-(C4) are immediate from Proposition \ref{ExplicitPBWCrystalFormula}. Conditions (S1) and (S2) are Corollary \ref{SaitoReflectionFormula}, and (S3)-(S4) follow from these using Proposition \ref{PBWStarProposition}.

All that remains is to check condition (I). That is, to show that, if $\lbc(PBW_b)=\lambda$, then
\begin{equation} 
\text{$\rbc_{\alpha_1}(PBW_b)= \frac{|\alpha_0|}{|\alpha_1|} \lambda_1$, $\rbc_\delta(PBW_{b})=\lambda \backslash \lambda_1$, $\rbc_{\alpha_0}(PBW_{b}) =\lambda_1$,}
\end{equation}
and $\rbc_\beta(PBW_b)=0$ for all other $\beta.$ 
We proceed by induction on $\lambda$, using the total order where $\lambda<\lambda'$ if
\begin{enumerate}
\item $|\lambda| < |\lambda'|$, or
\item $|\lambda| = |\lambda'|$ and $ (\lambda_1, \lambda_2, \ldots) >_\text{lex}   (\lambda'_1, \lambda'_2, \ldots)$.
\end{enumerate}
So, fix $\lambda$, and assume that (I) holds for all $\lambda' < \lambda$. In particular, we can apply Proposition \ref{rem:partial-char}.

It is easy to see that there is a unique $b \in B(-\infty)$ with the following MV polytope:  
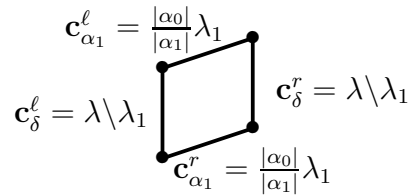
\begin{center}
\begin{tikzpicture}
\begin{scope}[xscale=-0.6, yscale=-0.2]
\draw

(0,0) node {$\bullet$}
(0,6) node {$\bullet$}
(2,2) node {$\bullet$}
(2,8) node {$\bullet$};

\draw[line width = 0.05cm] 
(0,0)--(2,2)--(2,8)--(0,6)--cycle;

\draw
(2.4,-0.6) node {$\lbc_{\alpha_1} = \frac{|\alpha_0|}{|\alpha_1|}\lambda_1$}
(3.8,5) node {$\lbc_\delta= \lambda \backslash \lambda_1$}
(-2,3.5) node {$\rbc_\delta=\lambda \backslash \lambda_1$}
(0,8.8) node {$\rbc_{\alpha_1}=\frac{|\alpha_0|}{|\alpha_1|}\lambda_1$};
\end{scope}
\end{tikzpicture}
\end{center}

\noindent That is,
\begin{equation} \label{eq:2c}
\cc^r(MV_b)_\delta = \cc^\ell(MV_b)_\delta=  \lambda \backslash \lambda_1, \quad  \cc^r(MV_b)_{\alpha_1} =  \cc^\ell(MV_b)_{\alpha_1} = \frac{|\alpha_0|}{|\alpha_1|} \lambda_1,
\end{equation}
and all other entries are $0$. By Proposition \ref{rem:partial-char}, $PBW_b$ has these same Lusztig data. That is, $b = b(\cc,0) = b(\cc,1)$, where $\cc= \lbc(MV_b)$.  Let $\dd$ be the Lusztig datum such that $b(\dd,0) = \e^{\lambda_1}_0 b(\cc,0)$. Using the action of $\tilde e_0$ on the $0$-PBW basis and the fact that the map from $B(-\infty)$ to PBW basis elements is injective, it suffices to show $\dd=\lambda$.

Using the upper triangularity of the PBW basis, 
\begin{align}
b(\cc,0) = E_1^{( \frac{|\alpha_0|}{|\alpha_1|} \lambda_1)}S_{\lambda \backslash \lambda_1} + \sum_{\cc^\prime \succ_0 \cc} a_{\cc^\prime,\cc} L(\cc^\prime,0).
\end{align}
Multiplying both sides on the left by $E_0^{(\lambda_1)}$, 
\begin{align}
E_0^{(\lambda_1)}b(\cc,0) = E_0^{(\lambda_1)}E_1^{(\frac{|\alpha_0|}{|\alpha_1|} \lambda_1)}S_{\lambda \backslash \lambda_1} + \sum_{\cc^\prime \succ_0 \cc} a_{\cc^\prime,\cc} E_0^{(\lambda_1)}L(\cc^\prime,0).
\end{align}
For each $\cc^\prime \succ_0 \cc$, rewrite $E_0^{(\lambda_1)} L(\cc^\prime,0)$ in the 0-PBW basis. Since $\cc'$ must have $\cc'_{\alpha_0+k \delta} \neq 0$ for some $k$, by Lemma \ref{ReorderingLemma}, no purely imaginary terms appear. When we subsequently expand these in $\cB$, by Theorem \ref{UpperTriangularityTheorem} we still don't get any purely imaginary terms.

Now expand $E_0^{(\lambda_1)}E_1^{(\frac{|\alpha_0|}{|\alpha_1|} \lambda_1)}S_{\lambda \backslash \lambda_1}$ in the basis $\cB$.  By Theorem \ref{UpperTriangularityTheorem} and Lemma \ref{TriangleLemma}, 
$$E_0^{(\lambda_1)}E_1^{(\frac{|\alpha_0|}{|\alpha_1|} \lambda_1)} = S_{(\lambda_1)} + \sum_{\cc^\prime \succ_0 (\lambda_1)} a_{\cc^\prime,(\lambda_1)} L(\cc^\prime,0),$$ 
and none of the $\cc^\prime$ that appear are purely imaginary. As before, the same remains true when we expand in the basis $\cB$.

By the Pieri rule (see e.g. \cite[Formula 4.14]{BCP}),
$$S_{(\lambda_1)} S_{\lambda \backslash \lambda_1} = S_\lambda + \sum_{\mu} S_\mu,$$ 
where the $\mu$ that appear all satisfy $\mu < \lambda$ in the order defined above. 

Suppose for $\dd \neq \lambda$. By Proposition \ref{CanonicalBasisCrystalFormula} the element $b(\dd,0)$ must show up with non-zero coefficient when $E_0^{(\lambda_1)}b(\cc,0)$ is written in the basis $\cB$. So either $\dd$ is not purely imaginary, or $\dd$ is purely imaginary and equal to $\mu$ with $\mu < \lambda$. 

If $\dd$ is not purely imaginary, then $|\dd_\delta| < |\lambda|$. So by Proposition \ref{rem:partial-char}, $\dd =  \lbc ( PBW_{\e^{\lambda_1}_0 b})=\lbc(MV_{\e^{\lambda_1}_0 b})= \lambda$, which is a contradiction.

If $\dd = \mu$ with $\mu < \lambda$, then (I) holds by induction.  We then see that $\rbc(PBW_{\e^{\lambda_1}_0 b})= \cc'$, where $\cc'_\delta = \mu \backslash \mu_1, \cc'_{\alpha_0} = \mu_1, \cc'_{\alpha_1} = \frac{|\alpha_0|}{|\alpha_1|} \mu_1$ and otherwise zero. But 
from \eqref{eq:2c} and the definition of $\tilde e_0$, we see that $\rbc(PBW_{\e^{\lambda_1}_0 b})= \cc''$, where $\cc''_\delta = \lambda \backslash \lambda_1, \cc''_{\alpha_0} = \lambda_1, \cc''_{\alpha_1} = \frac{|\alpha_0|}{|\alpha_1|} \lambda_1$, which is a contradiction.

\end{proof}

\section{Comparing combinatorial and geometric $\asl_2$ MV polytopes} \label{sec:quiver}

 \subsection{The $\asl_2$ quiver variety}
We will largely follow the conventions of \cite[\S7.4]{BKT}. Let $\tilde{Q}$ be the quiver
$$
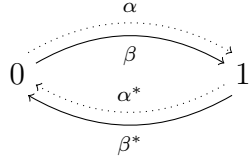
\begin{tikzpicture}
\node (0) at (0,0){$0$};
\node (1) at (3,0){$1$};
\node (0N) at (0,0.2){};
\node (0S) at (0,-0.2){};
\node (1N) at (3,0.2){};
\node (1S) at (3,-0.2){};
\draw[->,dotted] (0N) to[out=30,in=150]
  node[midway,above]{$\scriptstyle\alpha$} (1N);
\draw[->] (0) to[out=30,in=150]
  node[midway,below]{$\scriptstyle\beta$} (1);
\draw[->,dotted] (1) to[out=-150,in=-30]
  node[midway,above]{$\scriptstyle\alpha^*$} (0);
\draw[->] (1S) to[out=-150,in=-30]
  node[midway,below]{$\scriptstyle\beta^*$} (0S);
\end{tikzpicture}
$$

\setlength{\unitlength}{0.5cm}

\noindent Let $e_0$ and $e_1$ denote the lazy paths at the vertices $0$ and $1$ respectively. 
The preprojective algebra $\Pi$ is the quotient of the completed path algebra of $\tilde{Q}$ (completed with respect to the ideal generated by $\alpha, \alpha^*, \beta, \beta^*$) by the relations 
\begin{equation*}
\alpha \alpha^*+\beta \beta^*=0, \quad \alpha^* \alpha+\beta^* \beta=0.
\end{equation*}

A representation $T$ of $\Pi$ consists of a $\{0,1\}$-graded vector space $V = V_0 \oplus V_1$ and a $4$-tuple of linear operators $(t_\alpha,t_\beta: V_0 \rightarrow V_1,t_{\alpha^*},t_{\beta^*}: V_1 \rightarrow V_0)$ that satisfy 
$$t_\alpha t_{\alpha^*}+t_\beta t_{\beta^*}=0  \quad \text{  and  } \quad  t_{\alpha^*} t_\alpha+t_{\beta^*} t_\beta=0,$$ 
and which is nilpotent in the sense that, for some $N$ and any path $a_N \cdots a_1$ in $\tilde Q$, $t_{a_n} \cdots t_{a_1}=0$. 

Given an element $\mu = n \alpha_0 + m \alpha_1$ in the positive root lattice for $\asl_2$, let $\Pi(\mu)$ be variety of $\Pi$-representations on a fixed $\{0,1\}$-graded vector space $V^\mu = V^\mu_0 \oplus V^\mu_1$ with $\dim V^\mu_0 = n$ and $\dim V^\mu_1 = m$. We refer to $\mu$ as the dimension vector of $V^\mu$, and we will drop the superscripts $\mu$ when they are clear from context. 
Let $S_0$ and $S_1$ be the simple modules of dimension vectors $\alpha_0$ and $\alpha_1$ respectively (where all four maps $t_\alpha,t_\beta,t_{\alpha^*},t_{\beta^*}$ are $0$).

Let $\Irr \Pi(\mu)$ denote the set of irreducible components of $\Pi(\mu)$. Kashiwara and Saito \cite{KS:1997} show that 
\begin{equation}
 \coprod_{\mu} \Irr \Pi({\mu}),
 \end{equation} 
gives a realization of the crystal $B(-\infty)$, where the crystal operator $\e_i$ can be defined as follows:

For each $Z \in \Irr \Pi({\mu})$, there is a dense open subset $U \subset Z$ such that each $T \in U$ has $i$-cosocle of the same dimension $n$. For each $T \in U$, let $T' = \ker ( T \rightarrow S_i^{\oplus n})$. Let $W$ be the set of modules $T''$ fitting into a short exact sequence as below for some $T \in U$.  

$$ 0 \rightarrow T' \rightarrow  T'' \rightarrow S_i^{\oplus n+1} \rightarrow 0$$
It is known that there is a unique irreducible component  $Z' \in \Irr \Pi({\mu+\alpha_i})$ such that $W \cap Z'$ is dense in $Z'$. Kashiwara and Saito then define $\e_i Z =Z'$.

The map $\alpha \leftrightarrow \alpha^*$ and $\beta \leftrightarrow \beta^*$ extends uniquely to an algebra anti-involution of $\Pi$.
Given a $\Pi$-module $M$, the dual module is naturally a right $\Pi$ module, but we can twist the action by the above anti-involution to get a new left $\Pi$ module. We denote this left $\Pi$ module by $M^*$. Given $Z \in \Irr \Pi(\mu)$, 
$$\{ S \in \Pi(\mu) \mid S \simeq T^* \text{ for some } T \in Z \}$$ 
is also an irreducible component of $\Pi(\mu)$, which we denote by $Z^*$. In the above realization of $B(-\infty)$, the map $Z \rightarrow Z^*$ is Kashiwara's involution as discussed in \S\ref{ss:crystals}.

We also have an algebra automorphism $\tau$ of $\Pi$ defined on generators by the map $\alpha \leftrightarrow \alpha^*$, $\beta \leftrightarrow \beta^*$, and $e_0 \leftrightarrow e_1$. 
Twisting by $\tau$ induces an involutive auto-equivalence $R \rightarrow R^\tau$ on the category of left $\Pi$ modules and defines an involution on the set $\bigsqcup_{\mu} \Irr \Pi({\mu})$, inducing the $\asl_2$ diagram automorphism on $B(-\infty)$.

\subsection{Reflection functors and Harder-Narasimhan filtrations} 
Here we review the filtrations given in \cite[Theorems 5.11 and 5.12]{BKT}. We must first introduce the reflection functors $\Sigma_i$ and $\Sigma_i^*$ for $i \in \{0,1\}$ from \cite{BK??, BIRS09}.

\begin{Definition}
For $i \in \{0,1\}$ define the $\Pi-\Pi$ bimodule $I_i = \Pi(1-e_i)\Pi = \Pi e_{i+1} \Pi$ (subscripts taken modulo 2). If $s_{i_1} \cdots s_{i_k}$ is a reduced expression in the Weyl group, define $I_{s_{i_1} \cdots s_{i_k}} = I_{i_1} \otimes_\Pi \cdots \otimes_\Pi I_{i_k}$. 
\end{Definition}

As shown in \cite{BIRS09}, the bimodule $I_{s_{i_1} \cdots s_{i_k}}$ depends only on the Weyl group element $w = s_{i_1} \cdots s_{i_k}$ and not on the reduced expression. We need the following two endofunctors on the category of finite-dimensional $\Pi$ modules.

\begin{Definition}
$\Sigma_i = \text{Hom}_\Pi(I_i,?)$ and $\Sigma_i^* = I_i \otimes_\Pi ?$.
\end{Definition}
These functors are geometric lifts of Saito's crystal reflections in the following sense:
\begin{Proposition} \cite[Theorem 5.3]{BK??} \label{prop:functorial-Saito}
Fix $b \in B(-\infty)$ such that $\f_i(b)=0$. Let $Z_b$ and $Z_{\sigma_i(b)}$ be the irreducible components corresponding to $b$ and $\sigma_i(b)$ respectively, where $\sigma_i$ is Saito's reflection. For generic $T \in Z_b$, $\Sigma_i T$ is isomorphic to a point in $Z_{\sigma_i(b)}$, and furthermore this point in generic in the sense that the decorated Pseudo-Weyl polytope associated to $Z_{\sigma_i(b)}$ can be calculated using $\Sigma_i T$.

Similarly if $b \in B(-\infty)$ is such that $\f_i^*(b)=0$, then for generic $T \in Z_b$, $\Sigma_i^* T$ is isomorphic to a generic point in $Z_{\sigma_i^*(b)}$.
\qed
\end{Proposition}

To make this construction concrete in the case of $\asl_2$, it is convenient to introduce notation for the following special $\Pi$-modules.

\begin{Definition} \label{def:real-mods} 
For $k \geq 0$, 
\begin{enumerate}
 \item \label{ji1} $R^\ell(\alpha_1+ k \delta) = I_{s_1 \cdots s_k} \otimes_\Pi S_{k+1} = \Sigma^*_1 \cdots \Sigma^*_k S_{k+1}$
 \item \label{ji2} $R^\ell(\alpha_0 + k\delta) = \text{Hom}_\Pi (I_{s_{k-1} s_{k-2} \cdots s_0}, S_k) = \Sigma_0  \cdots \Sigma_{k-1} S_k$.
 \item \label{ji3} $R^r(\alpha_1+ k \delta) =  \text{Hom}_\Pi (I_{s_{k} s_{k-1} \cdots s_1}, S_{k+1}) = \Sigma_1  \cdots \Sigma_{k} S_{k+1}$
 \item \label{ji4}  $R^r(\alpha_0+ k \delta) =I_{s_0 \cdots s_{k-1}} \otimes_\Pi S_k = \Sigma^*_0 \cdots \Sigma^*_{k-1} S_k $.
\end{enumerate}
That the two definitions given on lines \eqref{ji1} and \eqref{ji4} agree follows by the definition of the reflection functors, and for \eqref{ji2} and \eqref{ji3} this follow by applying duality to  \eqref{ji1} and \eqref{ji2}. 
\end{Definition}

\noindent One can easily verify that these modules are as shown in Figure \ref{Fig:reps}, and in particular
$$\begin{aligned}
& \dim R^\ell(\alpha_1+ k \delta) =  \dim R^r(\alpha_1+ k \delta) = \alpha_1+ k \delta, \\
& \dim R^\ell(\alpha_0+ k \delta) =  \dim R^r(\alpha_0+ k \delta) = \alpha_0+ k \delta, \\
& R^r(\alpha_1+ k \delta) =R^\ell(\alpha_1+ k \delta)^*, \quad R^r(\alpha_0+ k \delta) =R^\ell(\alpha_0+ k \delta)^*, \\
& R^\ell(\alpha_0+ k \delta) =R^r(\alpha_1+ k \delta)^\tau, \quad R^\ell(\alpha_1+ k \delta) =R^r(\alpha_0+ k \delta)^\tau.
\end{aligned}
$$

\begin{figure}[ht]
\begin{center}

\setlength{\unitlength}{0.5cm}

\begin{tikzpicture}[scale=0.5]

\draw 
node at (0.5,2) {$0$}
node at (1.5,2) {$0$}
node at (3.5,2) {$0$}
node at (4.5,2) {$0$}
node at (2.5,2){$\cdots$}
node at (2.5,0.15){$\cdots$};

\draw[->] (0.4,1.6)--(0,0.4);
\draw[->] (1.4,1.6)--(1,0.4);
\draw[->] (3.4,1.6)--(3,0.4);
\draw[->] (4.4,1.6)--(4,0.4);

\draw[dotted,->] (0.6,1.6)--(1,0.4);
\draw[dotted,->] (1.6,1.6)--(2,0.4);
\draw[dotted,->] (3.6,1.6)--(4,0.4);
\draw[dotted,->] (4.6,1.6)--(5,0.4);

\draw
node at (0,0){$1$}
node at (1,0){$1$}
node at (4,0){$1$}
node at (5,0){$1$}
;

\draw node at (2.5,-1.5){$R^\ell(\alpha_1 + (j-2) \delta)$};

\end{tikzpicture}
\hspace{1in}
\begin{tikzpicture}[scale=0.5]

\draw
node at (2.5,2){$\cdots$}
node at (2.5,0.15){$\cdots$}
node at (0,2){$0$}
node at (1,2){$0$}
node at (4,2){$0$}
node at (5,2){$0$}
;

\draw[->] (0.9,1.6)--(0.5,0.4);
\draw[->] (1.9,1.6)--(1.5,0.4);
\draw[->] (3.9,1.6)--(3.5,0.4);
\draw[->] (4.9,1.6)--(4.5,0.4);

\draw[dotted,->] (0.1,1.6)--(0.5,0.4);
\draw[dotted,->] (1.1,1.6)--(1.5,0.4);
\draw[dotted,->] (3.1,1.6)--(3.5,0.4);
\draw[dotted,->] (4.1,1.6)--(4.5,0.4);

\draw 
node at (0.5,0){$1$}
node at (1.5,0){$1$}
node at (3.5,0){$1$}
node at (4.5,0){$1$}
;

\draw node at (2.5,-1.5){\small $R^\ell(\alpha_0 + (j-1) \delta)$};

\end{tikzpicture}

\end{center}

\caption{\label{Fig:reps} The representations from Definition \ref{def:real-mods}. In each case, the number of $0$ is $j$. Here the vertices represent basis elements, the dotted arrows represent matrix elements of $1$ for $t_\alpha$, and solid arrows represent matrix elements of $1$ for $t_\beta$, and all other matrix elements are $0$. }
\end{figure}
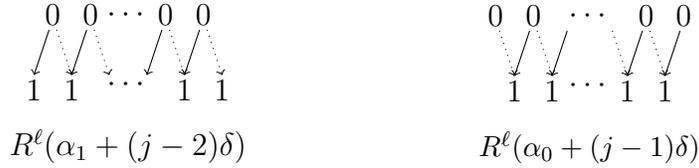

By \cite[Theorems 5.11 and 5.12]{BKT}, any finite dimensional representation $T$ of $\Pi$ admits a filtration
\begin{equation} \label{T:filt}
T = T^{\ell,0} \supset T^{\ell,1} \supset T^{\ell,2} \supset \cdots \supset T^{\ell,\infty} \supset T^\ell_\infty \supset \cdots \supset T^\ell_2 \supset T^\ell_1 \supset T^\ell_0 = 0 
\end{equation}
given by the following explicit formulas:

\begin{enumerate}

\item $T^{\ell,k} = \Sigma^*_1  \cdots \Sigma^*_k \Sigma_k \cdots \Sigma_1 T$

\item $T^\ell_k = \ker (T \rightarrow \Sigma_0 \cdots \Sigma_{k-1} \Sigma^*_{k-1} \cdots \Sigma^*_0 T).$ 

\item $T^\ell_\infty = \bigcup_k T^\ell_k$ and $T^{\ell,\infty} = \bigcap_k T^{\ell,k}$.

\end{enumerate}
These satisfy the following properties:
\begin{enumerate}
  \setcounter{enumi}{3}

\item For all $k$, $T^\ell_{k+1}/T^\ell_k$ is a direct sum of copies of $R^\ell(\alpha_0+k \delta)$. 

\item For all $k$, $T^{\ell,k}/T^{\ell,k+1}$ is a direct sums of copies of $R^\ell(\alpha_1+k \delta)$. 

\item \label{c:no-sub} No subrepresentation $S \subset T^\ell_\infty$ has $\langle \dim S, \alpha_0 \rangle >0$.

\item \label{c:no-quotient} No quotient representation $S$ of $T^{\ell,\infty}$ has $\langle \dim S, \alpha_0 \rangle <0$.
\end{enumerate}
There is also a filtration
\begin{equation} \label{T':filt}
T = T^{r,0} \supset T^{r,1} \supset T^{r,2} \supset \cdots \supset T^{r,\infty} \supset T^r_\infty \supset \cdots \supset T^r_2 \supset T^r_1 \supset T^r_0 = 0 
\end{equation}
given by:
\begin{enumerate}
\item $T^{r,k} =  \Sigma^*_0 \cdots \Sigma^*_{k-1} \Sigma_{k-1} \cdots \Sigma_0 T$

\item $T^r_k = \ker (T \rightarrow \Sigma_1 \cdots \Sigma_k \Sigma^*_k \cdots \Sigma^*_1 T)$ 

\item $T^r_\infty = \bigcup_k T^r_k$ and $T^{r,\infty} = \bigcap_k T^{r,k}$.

\end{enumerate}
which has the same properties as the first filtration, except the modules $R^\ell(\alpha_0+k \delta)$ and $R^\ell(\alpha_1+k \delta)$ are replaced with $R^r(\alpha_1+k \delta)$ and $ R^r(\alpha_0+k \delta)$ respectively, and $\alpha_1$ and $\alpha_0$ are interchanged in the above list of properties.

Following \cite[\S7.4]{BKT}, let $\Pi(n\delta)^\times$ be the subvariety of $\Pi(n\delta)$ consisting of those $4$-tuples of  operators $(t_\alpha,t_\beta,t_{\alpha^*},t_{\beta^*})$  where $t_\alpha$ is invertible. Define $I^\ell(n)$ to be the subvariety of $\Pi(n\delta)^\times$ where $t_{\beta^*} t_\alpha$ is nilpotent of order $n$ (i.e. $(t_{\beta^*} t_\alpha)^n = 0$, but $(t_{\beta^*}t_\alpha)^{n-1} \neq 0$), and notice that $I^\ell(n)$ consists only of indecomposable modules. 
By the discussion in \cite{BKT}, $I^\ell(n)$ is an open subset of an irreducible component of $\Pi(n\delta)$. Similarly, we define $I^r(n).$

\begin{Proposition} \cite[Proposition 7.11]{BKT} \label{prop:ima-part} Fix an irreducible component $Z$. There is a unique partition $\lambda^\ell = \lambda^\ell_1 \geq \cdots \geq \lambda^\ell_k$ such that, for all $T$ in some open dense subset of $Z$, $T^{\ell,\infty}/T^\ell_\infty$ can be decomposed as $\bigoplus_i T_{\lambda^\ell_i}$, where $ T_{\lambda^\ell_i} \in I^\ell(\lambda^\ell_i)$. 

Similarly there is a partition $\lambda^r$ such that, for a generic $T \in Z$, $T^{r,\infty}/T^r_\infty $ can be decomposed as $\bigoplus_i T_{\lambda^r_i}$, where $ T_{\lambda^r_i} \in I^r(\lambda^r_i)$.
\qed
\end{Proposition} 

\begin{Proposition}{\label{prop:quiverimagaxiom}}

Every element $T \in I^\ell(n)$ has socle $S_1$ and cosocle $S_0$. Furthermore $\ker ( T/S_1 \rightarrow S_0)$ is isomorphic to a point in $I^r(n-1)$.

\end{Proposition}

\begin{proof}

Let $W = t_\alpha( \ker(t_{\beta^*} t_\alpha))$. Since $t_{\beta^*} t_\alpha$ is nilpotent of order $n$, $W$ is the socle of $T$, and it is  isomorphic to $S_1$. Let $U = \text{coker} (t_{\beta^*} t_\alpha)$. Since $t_{\beta^*} t_\alpha$ has order $n$ and $t_\alpha$ is invertible, $U$ is the cosocle of $T$, and it is isomorphic to $S_0$. Moreover, we can easily check that the sub quotient $\ker(T/W \rightarrow U)$ has dimension vector $(n-1)\delta$, the induced operator $t_{\beta^*}$ is invertible, and $t_\alpha  t_{\beta^*}$ is  nilpotent of order $n-1$. So, $\ker(T/W \rightarrow U)$ lies in $I^r(n-1)$.
\end{proof}

\begin{Proposition} {\label{prop:quiverreflectionformula}}
The following hold:
\begin{enumerate}
\item $\Sigma_0^*R^\ell(\alpha_1 + k \delta) = R^r (\alpha_0 + (k+1) \delta)$ 
\item $\Sigma_0^*R^\ell(\alpha_0 + k \delta) = R^r (\alpha_1 + (k-1) \delta)$ for $k>0$.
\item If $T \in I^\ell(n)$,  then $\Sigma_0^* T \in I^r (n)$ for $n>0$.

\item $\Sigma_1^*R^r(\alpha_0 + k \delta) = R^\ell(\alpha_1 +(k+1) \delta)$
\item $\Sigma_1^*R^r(\alpha_1 + k \delta) = R^\ell(\alpha_0 +(k-1) \delta)$ for $k>0$.
\item If $T \in I^r(n)$, then $\Sigma_1^* T \in I^\ell(n)$ for $n>0$.

\item $\Sigma_0 R^r(\alpha_1 + k \delta) = R^\ell (\alpha_0 + (k+1) \delta)$ 
\item $\Sigma_0 R^r(\alpha_0 + k \delta) = R^\ell (\alpha_1 + (k-1) \delta)$ for $k>0$.
\item If $T \in I^r(n)$,  then $\Sigma_0 T \in I^\ell (n)$ for $n>0$.

\item $\Sigma_1 R^\ell(\alpha_0 + k \delta) = R^r(\alpha_1 +(k+1) \delta)$
\item $\Sigma_1 R^\ell(\alpha_1 + k \delta) = R^r(\alpha_0 +(k-1) \delta)$ for $k>0$.
\item If $T \in I^\ell(n)$, then $\Sigma_1T \in I^r(n)$ for $n>0$.

\end{enumerate}
\end{Proposition}

\begin{proof}
Statement (i) follows immediately from the definition. Since $R^r(\alpha_1 +k\delta)$ has no $0$-cosocle, $\Sigma_0^* \Sigma_0 R^r(\alpha_1 +k\delta) = R^r(\alpha_1 +k\delta)$ (see \cite[Equation 5.2]{BKT}), which implies (ii).

For (iii), let  $T \in I^\ell(n)$. A short calculation shows that $\tilde T = \Sigma_0^* T = (\tilde t_\alpha,\tilde t_\beta,\tilde t_{\alpha^*},\tilde t_{\beta^*})$  has the property that $\tilde t_{\beta^*}$ is invertible and $\tilde t_\alpha \tilde t_{\beta^*}$ is nilpotent of order $n$, so $\tilde T \in  I^r(n).$

Statements (iv) - (vi) follow from the first three by applying $\tau$. The remaining six statements follow from the first six by applying $*$.
\end{proof}

\begin{Lemma}{\label{lem:saitofiltration}}
The following hold: 

\begin{enumerate}

\item If $T^{\ell,0} = T^{\ell,1}$, then  $(\Sigma_1 T)^{r,k} = \Sigma_1T^{\ell,k+1}$ and $(\Sigma_1T^{\ell,k+1})/ (\Sigma_1 T^{\ell,k+2}) = \Sigma_1 (T^{\ell,k+1}/T^{\ell,k+2})$.

\item If $T^{r,0} = T^{r,1}$, then  $(\Sigma_0 T)^{\ell,k} = \Sigma_0T^{r,k+1}$ and  $(\Sigma_0T^{r,k+1})/ (\Sigma_0 T^{r,k+2}) = \Sigma_0 (T^{r,k+1}/T^{r,k+2})$.
\item If $T^\ell_0 = T^\ell_1$, then $(\Sigma^*_0 T)^r_{k} = \Sigma^*_0 T^\ell_{k+1}$ and $(\Sigma^*_0T^\ell_{k+1})/ (\Sigma^*_0 T^\ell_{k+2}) = \Sigma^*_0 (T^\ell_{k+1}/T^\ell_{k+2})$.

\item If $T^r_0 = T^r_1$, then $(\Sigma^*_1 T)^\ell_{k} = \Sigma^*_1T^r_{k+1}$ and $(\Sigma^*_1 T^r_{k+1})/ (\Sigma^*_1 T^r_{k+2}) = \Sigma^*_1 (T^r_{k+1}/T^r_{k+2})$.
\end{enumerate}
\end{Lemma}

\begin{proof}

Using the explicit formulas for the filtrations, we have $\Sigma_1T^{\ell,k+1} = \Sigma_1 \Sigma^*_1 (\Sigma_1 T)^{r,k}$. We always have a surjective map $  (\Sigma_1 T)^{r,k} \rightarrow \Sigma_1 \Sigma^*_1 (\Sigma_1 T)^{r,k}$, whose kernel is precisely the $1$-socle of  $(\Sigma_1 T)^{r,k}$. But $(\Sigma_1 T)^{r,k} \subset \Sigma_1 T$, and $\Sigma_1 T$ has vanishing $1$-socle because $T^{\ell,0} = T^{\ell,1}$. Thus the above map is an isomorphism, giving $(\Sigma_1 T)^{r,k} = \Sigma_1T^{\ell,k+1}$.

Because $\Sigma_1$ is left-exact, we have a injection $(\Sigma_1T^{\ell,k+1})/ (\Sigma_1 T^{\ell,k+2})  \rightarrow \Sigma_1 (T^{\ell,k+1}/T^{\ell,k+2})$, and the obstruction to this map being an isomorphism is an element of $\text{Ext}^1(I_1, T^{\ell,k+1})$.  By \cite[Remark 5.5 (ii)]{BKT}, the essential image of the functor $\Sigma^*_1$ is precisely those modules $M$ with  $\text{Ext}^1(I_1, M)=0 $. By construction, $T^{\ell,k+1}$ is in the image of $\Sigma^*_1$, so we have $\text{Ext}^1(I_1, T^{\ell,k+1})=0 $. This completes the proof of (i).

The second statement follows from the first by applying $\tau$, and the remaining two statements follow by duality using the canonical isomorphisms
\begin{equation}
T^\ell_k = (T^*/(T^*)^{r,k})^* \text{ and } T^r_k = (T^*/(T^*)^{l,k})^*. \qedhere
\end{equation}

\end{proof}

\subsection{MV polytopes from quiver varieties} \label{ss:asl2-proof}
Associate a Lusztig data $\cc^\ell$ to each $\Pi$-module $T$ by
\begin{itemize}

\item  $\cc^\ell_{\beta_k}$ is defined by $T^\ell_{k}/T^\ell_{k-1}\simeq R^\ell(\alpha_0+(k-1) \delta)^{\oplus \cc^\ell_{\beta_k}}$. 

\item $\cc^\ell_{\beta^k}$ is defined by $T^{\ell, k-1}/T^{\ell,k} \simeq R^\ell(\alpha_1+(k-1) \delta)^{\oplus \cc^\ell_{\beta^k}}$. 

\item Because the filtration in \eqref{T:filt} commutes with direct sums, every indecomposable summand of $T^{\ell, \infty}/T^\ell_\infty$ must have dimension $k \delta$ for some $k$. Then $\cc^\ell_{\delta}$ is defined to be the partition whose parts are these $k$. Notice that, if $T^{\ell, \infty}/T^\ell_\infty \in \Pi(n\delta)^\times$, then this agrees with $\lambda^\ell$ from Proposition \ref{prop:ima-part}.
\end{itemize}
Similarly define a Lusztig data $\cc^r$ via the same definition twisted by $\tau$.
By Definition \ref{DecoratedPseudoWeylPolytope} these are the left and right sides of a decorated pseudo-Weyl polytope $P_T$. 

The function which associates to every $T \in \coprod_v \Pi(v)$ the polytope $P_T$ is constructible. Thus on each irreducible component  of $\coprod_v \Pi(v)$, $P_T$ takes on a unique generic value, i.e. the constant value it takes on some open dense subset. For $b \in B(-\infty)$, we write $\HN_b$ for the generic value of $P_T$ on $Z_b$, the irreducible component corresponding to $b$. This pseudo-Weyl polytope is called the MV polytope for $b$ in \cite{BKT}. The notation $\HN$ stands for ``Harder-Narasimhan," since this polytope is constructed using Harder-Narasimhan filtrations.

\begin{Theorem} \label{th:sl2=sl2}
For any $b \in B(-\infty)$, $\HN_b$ and the MV polytope $MV_b$ defined \cite{BDKT:??} agree except the imaginary parts are transposed partitions.
\end{Theorem}

\begin{proof}
For each $b \in B(-\infty)$, let $\overline \HN_b$ be the pseudo-Weyl polytope obtained from $\HN_b$ be taking the transpose of the partitions decorating each vertical edge. It suffices to show that $b \rightarrow \overline \HN_b$ satisfies the conditions of Theorem \ref{th:unique-aff}.

Axiom (W) is clear. Axioms (C1) and (C2) follow from the definition of the crystal operators, and (C3) and (C4) follow by using the star operators. 

Conditions (S1), (S2), (S3), and (S4) follow immediately from Proposition \ref{prop:quiverreflectionformula} and Lemma \ref{lem:saitofiltration}.  

All that remains is to prove (I).  So, fix a component $Z$ such that, for generic $T \in Z$,  $T = T^{\ell,\infty}/T^\ell_\infty$. By Proposition \ref{prop:ima-part}, for generic $T \in Z$, there is a partition $\lambda^\ell$ such that $T = \bigoplus_i  T_{\lambda^\ell_i}$, where $ T_{\lambda^\ell_i} \in I^\ell(\lambda^\ell_i)$. Letting $\overline \lambda^\ell$ denote the transpose of $\lambda^\ell$, Proposition \ref{prop:quiverimagaxiom} implies $T$ has socle isomorphic to $S_1^{\oplus \overline \lambda^{\ell}_1}$, cosocle isomorphic to $S_0^{\oplus \overline \lambda^{\ell}_1}$, that $\ker \left(T / S_1^{\oplus \overline \lambda^{\ell}_1} \rightarrow S_0^{\oplus \overline \lambda^{\ell}_1}\right) \in T^{r,\infty}/T^r_\infty$, and that $\ker \left(T / S_1^{\oplus \overline \lambda^{\ell}_1}\rightarrow S_0^{\oplus \overline \lambda^{\ell}_1}\right) = \bigoplus_i  T_{\lambda^\ell_i-1}$, where $ T_{\lambda^\ell_i-1} \in I^r(\lambda^\ell_i-1)$. This is precisely the content of (I).
\end{proof}

\subsection{Characterization of symmetric affine MV polytopes}
Along with \cite[\S1.6 and \S7.6]{BKT}, Theorem \ref{th:sl2=sl2} allows one to characterization of MV polytopes in all symmetric affine types. In this section we state this precisely. This section is essentially a rewording of results in \cite{BKT}.

\begin{Definition}
Let $\g$ be a symmetric affine Kac-Moody algebra of rank $r+1$. A decorated pseudo-Weyl polytope for $\g$ is a convex polytope whose edges are all translates of integer multiples of roots, along with a partition $\lambda_F$ associated to each (possibly degenerate) $r$-face $F$ parallel to $\delta$, such that, for each edge $e$ which is a translate of $k \delta$, the sum of $|\lambda_F|$ over all faces $F$ incident to $e$ is $k$.
\end{Definition}

\begin{Remark} In \cite{BKT} the decorating partitions are indexed by chamber weights of an underlying finite-type root system; it is straightforward to see that these in turn index the possible $r$-faces parallel to $\delta$, so the above wording is equivalent. 
\end{Remark}

\begin{Theorem} \label{th:comb-poly}
Let $\g$ be a symmetric affine Kac-Moody algebra of rank $r+1$. The type $\g$ MV polytopes are exactly those decorated pseudo-Weyl polytopes $P$ such that each 2-face $S$ satisfies either:
\begin{enumerate}
\item The roots parallel to $S$ form a rank-2 root system (of type $\mathfrak{sl}_2 \times \mathfrak{sl}_2$ or $\mathfrak{sl}_3$), and $S$ is an MV polytope of that type.

\item The roots parallel to $S$ form a root system of type $\asl_2$. $S$ is a Minkowski sum of a smaller polytope $S'$ with the line segment $\sum_{S \subset F} \lambda_F.$ For each edge $e$ of $S'$ parallel to $\delta$, let $\lambda_e = \lambda_F$ for the unique $r$-face $F$ of $P$ which contains $e$ but does not contain $S$. Then $S'$ along with this decoration is a type $\asl_2$ MV polytope.  \qed
\end{enumerate} 
\end{Theorem}

\begin{Remark}
An extension of Theorem \ref{th:comb-poly} to all (not necessarily symmetric) affine cases is given in \cite[Theorem B]{TW}. 
\end{Remark}

\renewcommand{\theenumi}{\roman{enumi}}
\renewcommand{\labelenumi}{(\theenumi)}

\end{document}